\documentclass{amsart}

\usepackage{amsmath, amssymb, amsthm,enumitem,bbm,bm,pict2e}
\usepackage{indentfirst}

\theoremstyle{definition}
\newtheorem{theorem}{Theorem}
\numberwithin{theorem}{section}
\newtheorem{proposition}[theorem]{Proposition}
\newtheorem{lemma}[theorem]{Lemma}
\newtheorem{remark}[theorem]{Remark}
\newtheorem{example}[theorem]{Example}
\newtheorem{cor}[theorem]{Corollary}

\DeclareMathOperator{\id}{id}

\DeclareMathOperator{\dist}{dist}

\newcommand{\la}{\lambda}

\title{Vertices of FFLV polytopes}
\author{Evgeny Feigin}
\address{Evgeny Feigin:\newline
National Research University Higher School of Economics, Department of Mathematics,
Usacheva str. 6, 119048, Moscow, Russia,\newline
{\it and }\newline
Tamm Theory Division, Lebedev Physics Institute
}
\email{evgfeig@gmail.com} 

\author{Igor Makhlin}
\address{Igor Makhlin:\newline
Max Planck Institute for Mathematics, Vivatgasse 7, 53111, Bonn, Germany\newline
{\it and }\newline
National Research University Higher School of Economics, 
International Laboratory of Representation Theory and Mathematical Physics\\
Usacheva str. 6, 119048, Moscow, Russia,\newline
{\it and }\newline
Landau Institute for Theoretical Physics, prospekt Akademika Semenova 1A, 142432 Chernogolovka, Russia}
\email{imakhlin@mail.ru}

\date{}

\begin{document}
\begin{abstract}
FFLV polytopes describe monomial bases in irreducible representations of $\mathfrak{sl}_n$ and $\mathfrak{sp}_{2n}$.
We study various sets of vertices of FFLV polytopes. First, we consider the special linear case. We prove the locality of the set of vertices
with respect to the type $A$ Dynkin diagram. Then we describe all the permutation vertices and after that we 
describe all the simple vertices and prove that their number is equal to the large Schr\"oder number.    
Finally, we derive analogous results for symplectic Lie algebras.
\end{abstract}

\maketitle

\setlist[enumerate,1]{leftmargin=1.2em}

\section*{Introduction}
The goal of this paper is to study the vertices of FFLV polytopes in types $A$ and $C$ with ``FFLV'' standing for Feigin-Fourier-Littelmann-Vinberg.
The polytopes in question were introduced and studied in \cite{FFL1,FFL2,V} in order to describe
monomial bases in the PBW degenerate representations of simple Lie algebras $\mathfrak{sl}_n$ and 
$\mathfrak{sp}_{2n}$. In particular, the polytopes are labeled by a highest weight of the corresponding Lie 
algebra and the number of integer points is equal to the dimension of the corresponding irreducible
highest weight representation. FFLV polytopes turned out to be useful in representation theory \cite{FFL3,G1,G2},
algebraic geometry \cite{FFL3,Ki} and combinatorics \cite{ABS}. In particular, it has been shown that
the projective toric varieties attached to FFLV polytopes serve as toric degenerations of flag varieties.
Since the vertices of the Newton polytope of a toric variety carry important information on its
geometric and combinatorial properties, it is desirable to study the structure of the vertices of FFLV polytopes. 

We give descriptions of the set of all the vertices and of two of its subsets: permutation 
vertices and simple vertices. Our main tool is the combinatorics of Dyck paths of types $A$ and $C$.
In Section \ref{defsec} we formulate all the needed definitions in type $A$ and state the theorems.
In Section \ref{allsec} we describe the set of vertices of FFLV polytopes in type $A$. 
Section \ref{permsec} deals with the permutation vertices and in Section~\ref{simpsec} we describe the simple vertices.
Finally, in Section \ref{Csec} we study the symplectic case.

It is worth pointing out that FFLV polytopes bear deep similarities to Gelfand-Tsetlin polytopes \cite{GZ}. 
Polytopes in both families have their integer points parametrizing bases in irreducible representations, provide toric 
degenerations of flag varieties \cite{AB,KM,caldero,FFL3} and appear as Newton-Okounkov bodies \cite{A,Ka,FFL3}. Furthermore, the paper~\cite{ABS} 
establishes a deep combinatorial connection. 

In view of these similarities it is remarkable that the corresponding sets of vertices are rather straightforward to describe in the case of Gelfand-Tsetlin polytopes (see, for instance,~\cite{M}). This is not the case for FFLV polytopes as is shown in this paper.

\section{Type $A$: definitions and results}\label{defsec}

Consider the complex Lie algebra $\mathfrak{g}=\mathfrak{sl}_n$. Fix a Cartan decomposition $\mathfrak g=\mathfrak n^-\oplus\mathfrak h\oplus\mathfrak n^+$ and a dominant integral weight $\lambda\in\mathfrak h^*$. In~\cite{FFL1} a monomial basis 
in the irreducible $\mathfrak{sl}_n$-module $L_\lambda$ with highest weight $\lambda$ was constructed. We now give the definition of this basis.

First we define the combinatorial set $\Pi_\lambda$ that parametrizes the basis. This set is comprised of certain arrays of integers each containing ${n \choose 2}$ elements. Array $T$ consists of elements $T_{i,j}$ with $1\le i<j\le n$. We visualize $T$ as a number triangle in the following way:
\begin{center}
\begin{tabular}{ccccccc}
$T_{1,2}$ &&$ T_{2,3}$&& $\ldots$ && $T_{n-1,n}$\\
&$T_{1,3}$ &&$ \ldots$&& $T_{n-2,n}$ &\\
&&$\ldots$ &&$ \ldots$& &\\
&&&$T_{1,n}$ &&&
\end{tabular}
\end{center}
Thus a horizontal row contains all $T_{i,j}$ with a given difference $j-i$.

To specify when $T\in\Pi_\lambda$ the notion of a {\it Dyck path} is used. We understand a Dyck path to be a sequence of pairs of integers $((i_1,j_1),\ldots,(i_N,j_N))$ with $1\le i<j\le n$ such that $j_1-i_1=j_N-i_N=1$ (both lie in the top row) and $(i_{k+1},j_{k+1})$ is either $(i_k+1,j_k)$ or $(i_k,j_k+1)$ for any $1\le k\le N-1$ (either the upper-right or the bottom-right neighbor of $(i_k,j_k)$). 

Now let $\lambda$ have coordinates $(a_1,\ldots,a_{n-1})$ with respect to the basis of fundamental weights $\omega_i$. $\Pi_\lambda$ is then the set of all arrays $T$ consisting of nonnegative integers such that for any Dyck path $$d=((i_1,j_1),\ldots,(i_N,j_N))$$ one has $$T_{i_1,j_1}+\ldots+T_{i_N,j_N}\le a_{i_1}+a_{i_1+1}+\ldots+a_{i_N}.$$ Here on the right we have a sum of consecutive coordinates which we will denote $M(\lambda,d)$. For a number triangle $T$ and a Dyck path $d$ we will also denote the left-hand side above via $S(T,d)$.

The basis vector $v_T\in L_\lambda$ corresponding to $T\in\Pi_\lambda$ can now be defined explicitly. To do so for every positive root $\alpha$ we fix $f_\alpha\in\mathfrak n^-$ in the root space of $-\alpha$. 
We denote the simple roots $\alpha_1,\ldots,\alpha_{n-1}$ and for any pair of integers $1\le i<j\le n$ we write $\alpha_{i,j}=\alpha_i+\ldots+\alpha_{j-1}$ and $f_{i,j}=f_{\alpha_{i,j}}$. We then have $$v_T=\left(\prod_{i,j} f_{i,j}^{T_{i,j}}\right) v_0$$ where $v_0\in L_\lambda$ is a fixed highest weight vector and an arbitrary order of the factors $f_{i,j}^{T_{i,j}}$ is chosen in each $v_T$.

We now move on to define the {\it FFLV polytope} $P_\lambda$. The polytope is contained in $U=\mathbb R^{n \choose 2}$ with coordinates enumerated by pairs $1\le i<j\le n$. A point $x=(x_{i,j})$ in this space is visualized as a number triangle in the same exact fashion as the arrays comprising $\Pi_\lambda$. We have $x\in P_\lambda$ if and only if all $x_{i,j}\ge 0$ and for any Dyck path $d$ one has $$S(x,d)\le M(\lambda,d).$$ We see that $P_\lambda$ is indeed a convex polytope and $\Pi_\lambda$ is precisely the set of integer points therein.

All three of our main theorems in type $A$ are explicit combinatorial descriptions of various sets of vertices of $P_\lambda$. The first result characterizes the set of all vertices. Here we give the precise statement for a regular weight $\lambda$ (i.e. with all $a_i>0$), the case of a singular $\lambda$ is similar and will be discussed in the end of the next section. 

An important observation is that $P_\lambda$ is equal to the Minkowski sum $$P_{a_1\omega_1}+\ldots+P_{a_{n-1}\omega_{n-1}},$$ this follows directly from the definition. A general property of Minkowski sums of polytopes is that every vertex of a Minkowski sum can be uniquely expressed as the sum of vertices of the Minkowski summands. Consequently, each vertex $v$ of $P_\lambda$ may be uniquely expressed as $v=v_1+\ldots+v_{n-1}$ with $v_i$ being a vertex of $P_{a_i\omega_i }$. Therefore, the task of describing the vertices of $P_\lambda$ can be broken down into two steps:
\begin{enumerate}[label=\arabic*)]
\item describe the set of vertices of each $P_{a_i\omega_i}$ and
\item determine when a sum $v_1+\ldots+v_{n-1}$ with $v_i$ a vertex of $P_{a_i\omega_i}$ is indeed a vertex of $P_\lambda$.
\end{enumerate}

Our approach is to further break down the second step, namely into
\begin{enumerate}[label=\arabic*)]
\item for each $1\le i\le n-2$ determine when $v_i+v_{i+1}$ is a vertex of $P_{a_i\omega_i}+P_{a_{i+1}\omega_{i+1}}$ and
\item show that  $v_1+\ldots+v_{n-1}$ is a vertex of $P_\lambda$ whenever $v_i+v_{i+1}$ is a vertex of $P_{a_i\omega_i}+P_{a_{i+1}\omega_{i+1}}$ for every $1\le i\le n-2$.
\end{enumerate}
(As before, $v_i$ denotes a vertex of $P_{a_i\omega_i}$.)

Now note that $P_\lambda$ is the Minkowski sum of $P_{a_i\omega_i}+P_{a_{i+1}\omega_{i+1}}$ and all the $P_{a_l\omega_l}$ with $l\neq i,i+1$, which combined with step 2) above produces the following theorem.

\begin{theorem}\label{allverts}
Let $\lambda$ be regular and for all $1\le i\le n-1$ choose a vertex $v_i$ of $P_{a_i\omega_i}$. The point $v_1+\ldots+v_{n-1}$ is a vertex of $P_\lambda$ if and only if $v_i+v_{i+1}$ is a vertex of $P_{a_i\omega_i}+P_{a_{i+1}\omega_{i+1}}$ for every $1\le i\le n-2$.
\end{theorem}

To present our second result we introduce the following notation. For any $T\in\Pi_\lambda$ the vector $v_T$ is obviously a weight vector, we denote its weight 
by $\mu_T$. 

Now, for any element $w$ of the Weyl group $W\cong S_n$ there exists a unique $T_w\in\Pi_\lambda$ with $\mu_{T_w}=w\lambda$. We show that such a $T_w$ is necessarily a vertex of $P_\lambda$ and provide an explicit description of the set of these $T_w$. To do so we associate points in $U$ with certain sets of subsegments in $[1,n]$.

First, for a pair $1\le i<j\le n$ let $d^{i,j}$ denote the Dyck path $$((i,i+1),(i,i+2),\ldots,(i,j),(i+1,j),\ldots,(j-1,j))$$ which goes down and to the right until it reaches $(i,j)$ and after that continues up and to the right.

Now consider $E\subset R$, with $R$ being the set of subsegments in $[1,n]$ with integer endpoints and of positive length. Let $x(E)\in U$ be (uniquely) defined by
\begin{enumerate}[label=\arabic*)]
\item if $[i,j]\notin E$, then $x(E)_{i,j}=0$ and
\item if ${[i,j]\in E}$, then $$S(x(E),d^{i,j})=a_i+a_{i+1}+\ldots+a_{j-1}=M(\lambda,d^{i,j}).$$
\end{enumerate}

The characterization of vertices of the form $T_w$ (which we refer to as ``permutation vertices'') can now be given.

\begin{theorem}\label{permverts}
All the points $T_w,w\in W$ are vertices of $P_\lambda$. Moreover, they comprise the set of points of the form $x(E)$ where $E\subset R$ has the following property. For any two segments in $E$ with a nonempty intersection this intersection lies in $E$ as well.
\end{theorem}

\begin{example}\label{permex}
We illustrate Theorem~\ref{permverts} by the example $n=3$ and $\la=(a_1,a_2)=(1,1)$. For all six permutation vertices we write out the corresponding set $E\subset R$, the vertex $x(E)$ itself, the weight $\mu_{x(E)}$ and the permutation $w(E)\in W=S_3$ with $w(E)\la=\mu_{x(E)}$, i.e. $T_{w(E)}=x(E)$. The way $w(E)$ is computed is described in the proof of Lemma~\ref{injection}. To make the fact that $w(E)\la=\mu_{x(E)}$ more visible we express $\mu_{x(E)}$ in terms of the weights $\varepsilon_i$, also defined in the proof of Lemma~\ref{injection}.\vspace{5mm}\\
$E=\varnothing$\hfill$E=\{[1,2]\}$\vspace{2mm}\\
$x(E)=$\vspace{2mm}
\begin{tabular}{ccc}
0 &&0\\
&0 &
\end{tabular}
\hfill
$x(E)=$
\begin{tabular}{ccc}
1 &&0\\
&0 &
\end{tabular}\\
$\mu_{x(E)}=\la=\varepsilon_1-\varepsilon_3$\hfill$\mu_{x(E)}=\la-\alpha_{1,2}=\varepsilon_2-\varepsilon_3$\vspace{2mm}\\
$w(E)=\id$\hfill$w(E)=(1\: 2)$\vspace{7mm}\\
$E=\{[2,3]\}$\hfill$E=\{[1,3]\}$\vspace{2mm}\\
$x(E)=$\vspace{2mm}
\begin{tabular}{ccc}
0 &&1\\
&0 &
\end{tabular}
\hfill
$x(E)=$
\begin{tabular}{ccc}
0 &&0\\
&2 &
\end{tabular}\\
$\mu_{x(E)}=\la-\alpha_{2,3}=\varepsilon_1-\varepsilon_2$\hfill$\mu_{x(E)}=\la-2\alpha_{1,3}=-\varepsilon_1+\varepsilon_3$\vspace{2mm}\\
$w(E)=(2\: 3)$\hfill$w(E)=(1\: 3)$\vspace{7mm}\\
$E=\{[1,2],[1,3]\}$\hfill$E=\{[2,3],[1,3]\}$\vspace{2mm}\\
$x(E)=$\vspace{2mm}
\begin{tabular}{ccc}
1 &&0\\
&1 &
\end{tabular}
\hfill
$x(E)=$
\begin{tabular}{ccc}
0 &&1\\
&1 &
\end{tabular}\\
$\mu_{x(E)}=\la-\alpha_{1,2}-\alpha_{1,3}=-\varepsilon_1+\varepsilon_2$\hfill$\mu_{x(E)}=\la-\alpha_{2,3}-\alpha_{1,3}=-\varepsilon_2+\varepsilon_3$\vspace{2mm}\\
$w(E)=(1\: 3)(1\: 2)=[2\:3\:1]$\hfill$w(E)=(1\: 3)(2\: 3)=[3\:1\:2]$\\
\end{example}

Now let us recall that a polytope's vertex $v$ is said to be simple if the polytope's tangent cone at $v$ is simplicial and unimodular. Our third theorem describes the set of simple vertices of $P_\lambda$ in the case of a {\it regular} weight $\lambda$, which means that the stabilizer $W_\lambda$ is trivial or, equivalently, all $a_i>0$. 
\begin{theorem}\label{simpleverts}
Let $\lambda$ be regular. All the simple vertices of $P_\lambda$ are of the form $x(E)$ for some $E\subset R$. A point $x(E)$ is a simple vertex if and only if for 
any two segments in $E$ with a nonempty intersection this intersection coincides with one of the segments.
\end{theorem}

In particular, one sees that all the simple vertices are of the form $T_w$ for some $w\in S_n$. We prove the following
\begin{theorem}\label{simpleperm}
Let $\lambda$ be regular. Then a vertex $T_w$ is simple if and only if for any $1\le i<j\le n$ the condition 
$w^{-1}(j)\le i$ implies $w^{-1}(i+1)\le j$. 
In particular, the number of simple vetrices is equal to the $(n-1)$-st (large) Schr\"oder number.
\end{theorem}

\begin{example}
Obviously, when n=3, if two segments in $R$ intersect by more than a point, then one of them contains the other. Therefore, all the six vertices shown in Example~\ref{permex} are simple. This corresponds to the fact that the 2nd large Schr\"oder number is equal to 6. Examples of non-simple permutation vertices first appear for $n=4$. The 3rd large Schr\"oder number is equal to 22, therefore, for a regular $\la$ there are precisely two such vertices. For $\la=(a_1,a_2,a_3)=(1,1,1)$ they are as follows.\vspace{5mm}\\
$E=\{[2,3],[1,3],[2,4]\}$\hfill$E=\{[2,3],[1,3],[2,4],[1,4]\}$\vspace{2mm}\\
$x(E)=$\vspace{2mm}
\begin{tabular}{ccccc}
0&&1&& 0\\
&1&&1&\\
&&0 &&
\end{tabular}
\hfill
$x(E)=$
\begin{tabular}{ccccc}
0&&1&& 0\\
&1&&1&\\
&&1 &&
\end{tabular}\\
$w(E)=(2\:4)(1\:3)(2\:3)=[3\:1\:4\:2]$\hfill$w(E)=(1\:4)(2\:4)(1\:3)(2\:3)=[3\:4\:1\:2]$\\

By taking $i=1$ and $j=3$ one may verify that the two above $w(E)$ do not indeed possess the property described in Theorem~\ref{simpleperm}.
\end{example}

\begin{remark}
It has been proved in \cite{CFR2} that permutations $w$ satisfying the property from Theorem \ref{simpleperm}
correspond to the smooth torus fixed points in the degenerate flag varieties \cite{Fe1,Fe2,CFR2}.
So Theorem \ref{simpleperm} implies that the flat degeneration of the degenerate flag variety to the toric variety 
with the Newton polytope $P_\la$ (see \cite{FFL3}) provides a one-to-one correspondence between the smooth torus fixed points
and simple vertices of $P_\la$.  
\end{remark}

\section{The set of all vertices}\label{allsec}

The goal of this section is to describe the set of vertices of a polytope of the form $P_{a_i\omega_i+a_{i+1}\omega_{i+1}}$ and then prove Theorem~\ref{allverts}. From now and for the majority of this section we work in the assumption that $\lambda$ is regular, i.e. all $a_i$ are positive. After the case of a regular $\lambda$ has been considered, the case of a singular $\lambda$ can be dealt with easily which will be done in the end of this section.

Let us introduce a partial ordering of the set of pairs $(i,j)$ with $1\le i<j\le n$. We write $(i_1,j_1)\preceq (i_2,j_2)$ whenever $i_1\le i_2$ and $j_1\le j_2$. We will use $\prec$ to denote the corresponding strict relation. In particular, one sees that Dyck paths form chains in this partial order and there exists a Dyck path passing through both $(i_1,j_1)$ and $(i_2,j_2)$ if and only if they are comparable with respect to $\preceq$.

First, for an $1\le i\le n-1$ we discuss vertices of the polytope $P_{a_i\omega_i}=P_i$. Let $Q_i$ denote the set of pairs $(k,l)$ with $1\le k\le i$ and $i+1\le l\le n$.
\begin{proposition}\label{simproot}
A point constitutes a vertex of $P_i$ if and only if it has the form $a_i\chi_A$, where $\chi_A$ is the indicator function of an $\preceq$-antichain $A\subset Q_i$.
\end{proposition}
\begin{proof}
Since $P_i=a_i P_{\omega_i}$, it suffices to prove our assertion in the case $a_i=1$.

It is easy to see that for any $x\in P_{\omega_i}$ and any $(k,l)\not\in Q_i$ we have $x_{k,l}=0$. Consequently, $P_{\omega_i}$ can be viewed as a polytope in $\mathbb R^{Q_i}$. However, as a subset of $\mathbb R^{Q_i}$ the polytope $P_{\omega_i}$ is precisely the chain polytope of the poset $(Q_i,\preceq)$ as defined in~\cite{stan}. The proposition now follows as a direct application of Theorem 2.2 from~\cite{stan}.
\end{proof}

We move on to the characterization of the set of vertices of $$P_{a_i\omega_i+a_{i+1}\omega_{i+1}}=P_{i,i+1}$$ for an $1\le i\le n-2$. Every vertex of $P_{i,i+1}=P_i+P_{i+1}$ can be uniquely expressed as $v_i+v_{i+1}$ for vertices $v_i$ of $P_i$ and $v_{i+1}$ of $P_{i+1}$. Consequently, in accordance with Proposition~\ref{simproot}, a vertex of $P_{i,i+1}$ can be uniquely expressed as $a_i\chi_{A_i}+a_{i+1}\chi_{A_{i+1}}$ for $\preceq$-antichains $A_i\subset Q_i$ and $A_{i+1}\subset Q_{i+1}$.

We choose an arbitrary pair of antichains $A_i\subset Q_i$ and $A_{i+1}\subset Q_{i+1}$ and determine when $a_i\chi_{A_i}+a_{i+1}\chi_{A_{i+1}}$ constitutes a vertex of $P_{i,i+1}$. To do so, let us introduce a graph $\Gamma$ whose set of nodes is $A_i\cup A_{i+1}$. In this graph $(k_1,l_1)$ and $(k_2,l_2)$ are adjacent if and only if they are distinct and comparable with respect to $\preceq$. In other words, $\Gamma$ is the Hasse diagram of $(A_i\cup A_{i+1},\preceq)$, since the height of this latter poset is no greater than 2.

Now note that for any $$(k_1,l_1)\in Q_i\backslash Q_{i+1}=\{(\cdot,i+1)\}$$ and $$(k_2,l_2)\in Q_{i+1}\backslash Q_i=\{(i+1,\cdot)\}$$ we have $(k_1,l_1)\preceq (k_2,l_2)$. Moreover, being an antichain, $A_i$ contains at most one element from $Q_i\backslash Q_{i+1}$; similarly, $A_{i+1}$ contains at most one element from $Q_{i+1}\backslash Q_i$. This means that $\Gamma$ either has no nodes lying outside of $Q_i\cap Q_{i+1}$ or all (either one or two) of its nodes outside of $Q_i\cap Q_{i+1}$ belong to the same connected component of $\Gamma$. In the former case we define $\Delta=\varnothing$ and in the latter we define $\Delta\subset A_i\cup A_{i+1}$ to be the set of that connected component's nodes.

\begin{proposition}\label{tworoots}
In the above notation, the point $$v=a_i\chi_{A_i}+a_{i+1}\chi_{A_{i+1}}$$ constitutes a vertex of $P_{i,i+1}$ if and only if for any $(k,l)\in A_i\cup A_{i+1}$ we either have $(k,l)\in A_i\cap A_{i+1}$ or $(k,l)\in\Delta$.
\end{proposition}

\begin{example}\label{allex}
Before proving the proposition we illustrate it with several examples. We set n=5 and for each $1\le i\le 4$ we provide an example of an antichain $A_i\subset Q_i$ by displaying the corresponding number triangle $\chi_{A_i}$. The elements outside of the corresponding $Q_i$ are marked with an ``$\cdot$''.\vspace{2mm}\\
\small
\begin{tabular}{p{0.5mm}p{0.5mm}p{0.5mm}p{0.5mm}p{0.5mm}p{0.5mm}p{0.5mm}}
&&&{\normalsize $A_1$}&&&\vspace{1mm}\\
0 &&$\cdot$&& $\cdot$ && $\cdot$\\
&1 && $\cdot$&& $\cdot$ &\\
&&0&& $\cdot$& &\\
&&&0 &&&
\end{tabular}
\hspace{2mm}
\begin{tabular}{p{0.5mm}p{0.5mm}p{0.5mm}p{0.5mm}p{0.5mm}p{0.5mm}p{0.5mm}}
&&&\normalsize{$A_2$}&&&\vspace{1mm}\\
$\cdot$ &&0&& $\cdot$ && $\cdot$\\
&0 && 1&& $\cdot$ &\\
&&0&& 0& &\\
&&&1 &&&
\end{tabular}
\hspace{2mm}
\begin{tabular}{p{0.5mm}p{0.5mm}p{0.5mm}p{0.5mm}p{0.5mm}p{0.5mm}p{0.5mm}}
&&&\normalsize{$A_3$}&&&\vspace{1mm}\\
$\cdot$ &&$\cdot$&& 1 && $\cdot$\\
&$\cdot$&& 0&& 0 &\\
&&0&& 0& &\\
&&&1 &&&
\end{tabular}
\hspace{2mm}
\begin{tabular}{p{0.5mm}p{0.5mm}p{0.5mm}p{0.5mm}p{0.5mm}p{0.5mm}p{0.5mm}}
&&&\normalsize{$A_4$}&&&\vspace{1mm}\\
$\cdot$ &&$\cdot$&& $\cdot$ && 0\\
&$\cdot$&& $\cdot$&& 1 &\\
&&$\cdot$&& 0& &\\
&&&0 &&&
\end{tabular}
\vspace{2mm}

\normalsize
Let us display the three corresponding graphs~$\Gamma$ with the large dots marking the vertices of $\Gamma$, the smaller dots marking pairs $(i,j)$ not in $\Gamma$ and the component $\Delta$ labeled.
 
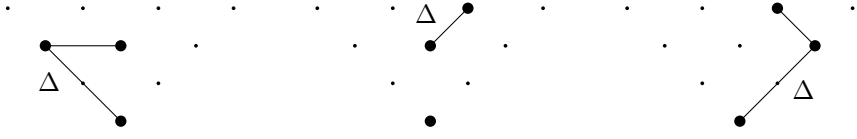
\begin{picture}(40,35)
\put(0,30){\circle*{0.5}}
\put(10,30){\circle*{0.5}}
\put(20,30){\circle*{0.5}}
\put(30,30){\circle*{0.5}}
\put(5,25){\circle*{1.5}}
\put(15,25){\circle*{1.5}}
\put(25,25){\circle*{0.5}}
\put(10,20){\circle*{0.5}}
\put(20,20){\circle*{0.5}}
\put(15,15){\circle*{1.5}}
\put(5,25){\line(1,0){10}}
\put(5,25){\line(1,-1){10}}
\put(4,19){$\Delta$}
\end{picture}
\begin{picture}(40,35)
\put(0,30){\circle*{0.5}}
\put(10,30){\circle*{0.5}}
\put(20,30){\circle*{1.5}}
\put(30,30){\circle*{0.5}}
\put(5,25){\circle*{0.5}}
\put(15,25){\circle*{1.5}}
\put(25,25){\circle*{0.5}}
\put(10,20){\circle*{0.5}}
\put(20,20){\circle*{0.5}}
\put(15,15){\circle*{1.5}}
\put(20,30){\line(-1,-1){5}}
\put(13,28){$\Delta$}
\end{picture}
\begin{picture}(40,35)
\put(0,30){\circle*{0.5}}
\put(10,30){\circle*{0.5}}
\put(20,30){\circle*{1.5}}
\put(30,30){\circle*{0.5}}
\put(5,25){\circle*{0.5}}
\put(15,25){\circle*{0.5}}
\put(25,25){\circle*{1.5}}
\put(10,20){\circle*{0.5}}
\put(20,20){\circle*{0.5}}
\put(15,15){\circle*{1.5}}
\put(20,30){\line(1,-1){5}}
\put(15,15){\line(1,1){10}}
\put(22,18){$\Delta$}
\end{picture}\vspace{-1.2cm}
We see that we have chosen the $A_i$ in such a way that, by Proposition~\ref{tworoots}, for each $1\le i\le 3$ the point $a_i\chi_{A_i}+a_{i+1}\chi_{A_{i+1}}$ is indeed a vertex of $P_{i,i+1}$. To produce an example of a non-vertex we consider $A'_2$ given by
\begin{center}
\begin{tabular}{ccccccc}
$\cdot$ &&0&& $\cdot$ && $\cdot$\\
&0 && 0&& $\cdot$ &\\
&&0&& 1& &\\
&&&0&&&
\end{tabular}
\end{center}
The graph $\Gamma$ corresponding to $A'_2$ and $A_3$ will then have the form

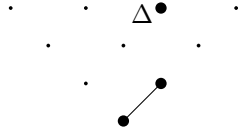
\begin{center}
\begin{picture}(40,35)
\put(0,30){\circle*{0.5}}
\put(10,30){\circle*{0.5}}
\put(20,30){\circle*{1.5}}
\put(30,30){\circle*{0.5}}
\put(5,25){\circle*{0.5}}
\put(15,25){\circle*{0.5}}
\put(25,25){\circle*{0.5}}
\put(10,20){\circle*{0.5}}
\put(20,20){\circle*{1.5}}
\put(15,15){\circle*{1.5}}
\put(20,20){\line(-1,-1){5}}
\put(16,28){$\Delta$}
\end{picture}
\end{center}
\vspace{-1.2cm}
We see that it has two nodes outside of both $\Delta$ and $A'_2\cap A_3$.
\end{example}

\begin{proof}[Proof of Proposition~\ref{tworoots}]
First we prove the ``if'' part. The point $v$ is obviously contained in $P_{i,i+1}$. We are to show that if a point $x\in P_{i,i+1}$ satisfies 
\begin{enumerate}[label=\arabic*)]
\item $x_{k,l}=0$ whenever $v_{k,l}=0$ (i.e. $(k,l)\not\in A_i\cup A_{i+1}$) and 
\item $S(x,d)=S(v,d)$ for every Dyck path $d$ such that $$S(v,d)=M(a_i\omega_i+a_{i+1}\omega_{i+1},d),$$ 
\end{enumerate}
then $x=v$.

Indeed, consider such a point $x$. All of its coordinates outside of $A_i\cup A_{i+1}$ coincide with those of $v$ by assumption, consider some $(k,l)\in A_i\cup A_{i+1}$. 

If $(k,l)\not\in\Delta$, we have $(k,l)\in A_i\cap A_{i+1}$ and $v_{k,l}=a_i+a_{i+1}$. For a Dyck path $d$ containing $(k,l)$ we have $$M(a_i\omega_i+a_{i+1}\omega_{i+1},d)=a_i+a_{i+1}$$ since $(k,l)\in Q_i\cap Q_{i+1}$. We also have $x_{k',l'}=0$ for any $(k',l')\in d$ other than $(k,l)$ due to property 1) above. That is since $(k',l')$ and $(k,l)$ are $\preceq$-comparable which means that $(k',l')$ may not lie in $A_i\cup A_{i+1}$. In view of property 2) we must now have $$x_{k,l}=a_i+a_{i+1}=v_{k,l}.$$

Suppose we have $(k,l)\in \Delta$. We choose an element $$(k_0,l_0)\in\Delta\backslash(Q_i\cap Q_{i+1})$$ and prove that $x_{k,l}=v_{k,l}$ by induction on the distance between $(k,l)$ and $(k_0,l_0)$ in $\Gamma$. To obtain the base consider the path $d^{k_0,l_0}$. If $l_0=i+1$ we have $$v_{k_0,l_0}=a_i=M(a_i\omega_i+a_{i+1}\omega_{i+1},d^{k_0,l_0}).$$ This means that $d^{k_0,l_0}$ may not intersect $A_i\cup A_{i+1}$ in any elements other than $(k_0,l_0)$. From properties 1) and 2) we deduce $$x_{k_0,l_0}=M(a_i\omega_i+a_{i+1}\omega_{i+1},d^{k_0,l_0})=v_{k_0,l_0}.$$ The case $k_0=i+1$ is considered analogously.

Let us move on to the step of our induction. We may choose $(k',l')\in\Delta$ with $$\dist((k_0,l_0),(k',l'))=\dist((k_0,l_0),(k,l))-1$$ and adjacent to $(k,l)$. The elements $(k,l)$ and $(k',l')$ are $\preceq$-comparable, consider a Dyck path $d$ containing both of them. We have $$v_{k,l}+v_{k',l'}=a_i+a_{i+1}=M(a_i\omega_i+a_{i+1}\omega_{i+1},d).$$ Consequently, all elements of $d$ other than $(k,l)$ and $(k',l')$ lie outside of $A_i\cup A_{i+1}$. Properties 1) and 2) then imply $$x_{k,l}+x_{k',l'}=M(a_i\omega_i+a_{i+1}\omega_{i+1},d)=v_{k,l}+v_{k',l'}$$ and we are left to employ the induction hypothesis for $(k',l')$.

To prove the ``only if'' part suppose that some $(k,l)$ is contained in exactly one of $A_i$ and $A_{i+1}$ and not contained in $\Delta$. Then the connected component $K$ of $\Gamma$ containing $(k,l)$ lies wholly within $Q_i\cap Q_{i+1}$. Define $$A'_i=(A_i\backslash K)\cup(A_{i+1}\cap K)$$ and $$A'_{i+1}=(A_{i+1}\backslash K)\cup(A_{i}\cap K),$$ i.e. define new antichains by transferring all elements in $K$ from $A_i$ to $A_{i+1}$ or vice versa. 

Let $x=a_i\chi_{A'_i}+a_{i+1}\chi_{A'_{i+1}}$. It is easy to see that $x$ possesses properties 1) and 2) above. As we have shown, this means that either $v$ is not a vertex or we have $x=v$. In the latter case, however, the point $v$ can be expressed as a sum of points in $P_i$ and $P_{i+1}$ in two different ways which, again, shows that $v$ is not a vertex.
\end{proof}

\begin{cor}\label{totheleft}
Let antichains $A_i\subset Q_i$, $A_{i+1}\subset Q_{i+1}$ be such that $a_i\chi_{A_i}+a_{i+1}\chi_{A_{i+1}}$ constitutes a vertex of $P_{i,i+1}$. Then for any two $\preceq$-comparable elements $(k_1,l_1)\in A_i$ and $(k_2,l_2)\in A_{i+1}$ we have $(k_1,l_1)\preceq (k_2,l_2)$.
\end{cor} 
\begin{proof}
From Proposition~\ref{tworoots} we deduce that in this case we must either have $(k_1,l_1)=(k_2,l_2)$ or $(k_1,l_1),(k_2,l_2)\in\Delta$. The former case is trivial, to deal with the latter we consider the element $(k_0,l_0)$ from the proof of Proposition~\ref{tworoots} and proceed by induction on the quantity $$\min\nolimits\big(\dist((k_0,l_0),(k_1,l_1)),\dist((k_0,l_0),(k_2,l_2))\big).$$ (As before we are considering distances in the graph $\Gamma$.)

As the base of our induction we either have $(k_1,l_1)=(k_0,l_0)$ or $(k_2,l_2)=(k_0,l_0)$ and the corollary 
follows from our choice of $(k_0,l_0)$.

For the step, suppose that $$\dist((k_0,l_0),(k_1,l_1))\le \dist((k_0,l_0),(k_2,l_2)).$$ Then we may choose an element $(k',l')\in A_i\cup A_{i+1}$ adjacent to $(k_1,l_1)$ and such that $$\dist((k_0,l_0),(k',l'))=\dist((k_0,l_0),(k_1,l_1))-1.$$ We must have $(k',l')\in A_{i+1}$ and, by the induction hypothesis, $(k_1,l_1)\preceq(k',l')$. Now we see that $(k_1,l_1)\succ(k_2,l_2)$ would lead to $(k_2,l_2)\prec(k',l')$ which is impossible. The case $$\dist((k_0,l_0),(k_1,l_1))>\dist((k_0,l_0),(k_2,l_2))$$ is considered analogously.
\end{proof}

We move on to the proof of Theorem~\ref{allverts}. Every vertex of $P_\lambda$ can be uniquely expressed as $a_1\chi_{A_1}+\ldots+a_{n-1}\chi_{A_{n-1}}$ for antichains $A_1\subset Q_1,\ldots, A_{n-1}\subset Q_{n-1}$. We consider an arbitrary sequence of antichains $(A_i\subset Q_i,1\le i\le n-1)$ and the point $$v=a_1\chi_{A_1}+\ldots+a_{n-1}\chi_{A_{n-1}}.$$ Let us make a simple observation.
\begin{proposition}\label{goodpath}
In this notation, for a Dyck path $$d=((i_1,j_1),\ldots,(i_N,j_N))$$ we have $S(v,d)=M(\lambda,d)$ if and only if $d$ intersects each of the antichains $A_{i_1},\ldots,A_{i_N}$ in one element. 
\end{proposition}
\begin{proof}
Visibly, $d$ does not intersect any $Q_i$ with $i<i_1$ or $i>i_N$ and intersects any $A_i$ in at most one element. Therefore, $S(v,d)$ is equal to the sum of $a_i$ over those $i_1\le i\le i_N$ for which $d$ intersects $A_i$. The proposition follows.
\end{proof}

\begin{theorem}[equivalent to Theorem~\ref{allverts}]\label{allequiv}
In the above notation, the point $v$ constitutes a vertex of $P_\lambda$ if and only if for all $1\le i\le n-2$ the point $a_i\chi_{A_i}+a_{i+1}\chi_{A_{i+1}}$ constitutes a vertex of $P_{i,i+1}$.
\end{theorem}
\begin{proof}
As we have pointed out, the ``only if'' part follows from general properties of Minkowski sums. We prove the ``if'' part.

Indeed, suppose that for all $1\le i\le n-2$ the point $a_i\chi_{A_i}+a_{i+1}\chi_{A_{i+1}}$ constitutes a vertex of $P_{i,i+1}$. We define a directed graph $\Theta$ whose set of nodes is $\bigcup A_i$ with an edge going from $(k_1,l_1)$ to $(k_2,l_2)$ if and only if for some $i$ we have $(k_1,l_1)\in A_i$, $(k_2,l_2)\in A_{i+1}$ and $(k_1,l_1)\prec (k_2,l_2)$.

As a running example let us consider $n=5$ and $A_i$, $1\le i\le 4$ as in Example~\ref{allex}. The corresponding graph $\Theta$ has the following form (with the nodes enumerated for later use).

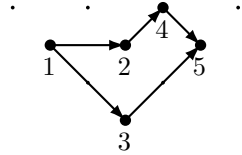
\begin{center}
\begin{picture}(40,35)
\put(0,30){\circle*{0.5}}
\put(10,30){\circle*{0.5}}
\put(20,30){\circle*{1.5}}
\put(19,26){4}
\put(30,30){\circle*{0.5}}
\put(5,25){\circle*{1.5}}
\put(4,21){1}
\put(15,25){\circle*{1.5}}
\put(14,21){2}
\put(25,25){\circle*{1.5}}
\put(24,21){5}
\put(10,20){\circle*{0.5}}
\put(20,20){\circle*{0.5}}
\put(15,15){\circle*{1.5}}
\put(14,11){3}
\thicklines
\put(15,25){\vector(1,1){5}}
\put(5,25){\vector(1,0){10}}
\put(5,25){\vector(1,-1){10}}
\put(20,30){\vector(1,-1){5}}
\put(15,15){\vector(1,1){10}}
\end{picture}
\end{center}
\vspace{-11mm}

Next, note that a Dyck path $$d=((i_1,j_1),\ldots,(i_N,j_N))$$ may intersect $A_{i_1}$ only in an element of the form $(i_1,\cdot)$ and may intersect $A_{i_N}$ only in an element of the form $(\cdot,j_N)$. This prompts us to introduce the following terminology: we refer to an element $(i,j)$ (node 
of $\Theta$) as a ``source'' if it belongs to $A_i$ and we refer to it as a ``destination'' if it belongs to $A_j$. In our running example nodes 1, 2 and 4 are sources while nodes 4 and 5 are destinations.

Let us establish several properties of $\Theta$. First of all, from Corollary~\ref{totheleft} we deduce that $\Theta$ is the directed Hasse diagram of $(\bigcup A_i,\preceq)$, i.e. there is an edge going from $(k_1,l_1)$ to $(k_2,l_2)$ if and only if $(k_2,l_2)$ covers $(k_1,l_1)$ in $(\bigcup A_i,\preceq)$. Now consider a path $$p=((i_1,j_1),\ldots,(i_N,j_N))$$ in $\Theta$ starting in a source and ending in a destination (we will refer to such paths as ``complete''). We see that $p$ may be extended to a Dyck path $d$ starting in $(i_1,i_1+1)$ and ending in $(j_N-1,j_N)$. Proposition~\ref{goodpath} implies that we will necessarily have $S(v,d)=M(\lambda,d)$ for any such $d$.

Conversely, if for a Dyck path $d$ we have $S(v,d)=M(\lambda,d)$, then the sequence $d\cap\bigcup A_i$ forms a complete path in $\Theta$. This also follows from Corollary~\ref{totheleft} and Proposition~\ref{goodpath}.

Furthermore, for $(k,l)\in\bigcup A_i$ not a source denote $i_0>k$ the least number such that $(k,l)\in A_{i_0}$. Applying Proposition~\ref{tworoots} to $A_{i_0-1}$ and $A_{i_0}$ we deduce that there is an edge in $\Theta$ going to $(k,l)$. Symmetrically, if $(k,l)$ is not a destination, then there is an edge in $\Theta$ going from $(k,l)$. These two facts imply that every vertex in $\Theta$ lies on a complete path.

The properties of $\Theta$ established in the previous three paragraphs can straightforwardly be verified for our running example.

Now, to prove that $v$ is a vertex we are to show that for a point $x$ satisfying
\begin{enumerate}[label=\arabic*)]
\item $x_{k,l}=0$ whenever $v_{k,l}=0$ (i.e. $(k,l)\not\in\bigcup A_i$) and 
\item $S(x,d)=S(v,d)$ for every Dyck path $d$ such that $S(v,d)=M(\lambda,d)$,
\end{enumerate}
we have $x=v$.

Indeed, consider such a point $x$. For a $(k,l)\in\bigcup A_i$ we are to show that $x_{k,l}=v_{k,l}$. In view of property 2) it suffices to prove that the indicator function $\chi_{\{(k,l)\}}$ may be expressed as a linear combination of the functions $\chi_{d\cap\bigcup A_i}$ over $d$ such that $S(v,d)=M(\lambda,d)$. In other words, we are to prove the following property of graph $\Theta$: the indicator function of any vertex in the graph can be expressed as a linear combination of indicator functions of complete paths.

Our approach to this latter proof is as follows. We introduce an order relation (``lies above'') on the set of paths in $\Theta$ and prove the statement for all $(k,l)\in p$ by induction on complete path $p$ with respect to our order relation. 

For path $p_1$ in $\Theta$ observe that the set of $i$ such that $p_1$ intersects $A_i$ constitutes a segment $[i_1,j_1]$. Consider a second path $p_2$ and the corresponding segment $[i_2,j_2]$. We say that $p_1$ lies above $p_2$ if and only if the following holds. First, we have $[i_1,j_1]\subset[i_2,j_2]$ and, second, for any $i\in[i_1,j_1]$ and elements $(k_1,l_1)=p_1\cap A_i$, $(k_2,l_2)=p_2\cap A_i$ we have $$l_1-k_1\le l_2-k_2,$$ i.e. $(k_1,l_1)$ lies in the higher or same row. (Note that if they lie in the same row then we must have $(k_1,l_1)=(k_2,l_2)$ due to $A_i$ being an antichain.)

We have indeed defined a non-strict partial order relation, we will term the corresponding strict relation ``lies strictly above''. Below, while establishing the induction step we will show that if a complete path $p$ contains at least two nodes, then there exists a complete path lying strictly above $p$. Therefore, minimal elements (lying above any comparable path) in the set of complete paths are those consisting of a single node and the base of our induction is trivial.

Now let us make the following observation. Consider some $(k,l)\in A_i$ and assume that $(k,l)\not\in A_{i-1}$. All nodes $(k',l')$ with an edge going from $(k',l')$ to $(k,l)$ belong to $A_{i-1}$. If $(k,l)$ is not a source, then among such $(k',l')$ we may choose a node for which the difference $l'-k'$ is the least and denote this node $\mathfrak l((k,l))$. In other words, $\mathfrak l((k,l))$ is the unique highest (by row) node with an edge from it to $(k,l)$. 

Next, consider the sequence $$q=\left((k,l),\mathfrak l((k,l)),\mathfrak l(\mathfrak l((k,l))\right),\ldots)$$ ending with its first source. We see that by reversing $q$ we obtain a path $\mathcal L((k,l))$ in $\Theta$ which starts in a source and ends in $(k,l)$ and lies above any other path with these two properties.

Symmetrically, we define the unique highest node $\mathfrak r((k,l))$ with an edge to it from $(k,l)$ (whenever $(k,l)$ is not a destination) and the path $\mathcal R((k,l))$ which starts in $(k,l)$ and ends in a destination and lies above any other path with these two properties.

To illustrate, in our example we have $\mathcal L(x_3)=(x_1,x_3)$ and $\mathcal R(x_3)=(x_3,x_5)$ where $x_i$ denotes the pair corresponding to node $i$. We also have $\mathcal L(x_5)=(x_4,x_5)$, $\mathcal L(x_4)=\mathcal R(x_4)=(x_4)$ and $\mathcal R(x_1)=(x_1,x_2,x_4)$.

We proceed to the induction step. Consider a complete path $p$ of at least two nodes. Let $(k_1,l_1)\prec (k_2,l_2)$ be two consecutive nodes in $p$. We show that there exists a complete path $p'$ lying strictly above $p$ such that one of the two holds. Either $p'$ contains all elements of $p$ preceding or equal to $(k_1,l_1)$ or $p'$ contains all elements of $p$ succeeding or equal to $(k_2,l_2)$.

This will suffice since it implies that all elements of $p$ other than at most one are contained in complete paths lying strictly above $p$. Due to the induction hypothesis the indicator functions of those elements can be expressed as linear combinations of indicator functions of complete paths. The indicator function of the remaining element (if it exists) can now also be expressed in this fashion since it is the difference of $\chi_p$ and the sum of indicator functions of all the other elements.

Indeed, there exists an $i$ such that $p\cap A_i=(k_1,l_1)$ and $p\cap A_{i+1}=(k_2,l_2)$. We consider four possible cases.
\begin{enumerate}[label=\arabic*.]
\item $(k_1,l_1)$ is a destination, i.e. $l_1=i+1$. Then we may take $p'$ to be $p$ without its elements equal to or succeeding $(k_2,l_2)$.
\item $(k_2,l_2)$ is a source, i.e. $k_2=i+1$. Then we may take $p'$ to be $p$ without its elements preceding or equal to $(k_1,l_1)$.
\end{enumerate}
If neither of the above holds, we apply Proposition~\ref{tworoots} to antichains $A_i$ and $A_{i+1}$ and consider the corresponding set $\Delta$. We see that $(k_1,l_1)$ and $(k_2,l_2)$ both belong to $\Delta$. It isn't hard to see that this is only possible when either $\mathfrak r((k_1,l_1))\in A_{i+1}$ lies in a higher row than $(k_2,l_2)$ or $\mathfrak l((k_2,l_2))\in A_i$ lies in a higher row than $(k_1,l_1)$. We obtain the remaining two possibilities.
\begin{enumerate}[label=\arabic*.]
\setcounter{enumi}{2}
\item $\mathfrak r((k_1,l_1))$ lies in a higher row than $(k_2,l_2)$. Then we may take $p'$ to be $p$ with its elements equal to or succeeding $(k_1,l_1)$ replaced with $\mathcal R((k_1,l_1))$.
\item $\mathfrak l((k_2,l_2))$ lies in a higher row than $(k_1,l_1)$. Then we may take $p'$ to be $p$ with its elements preceding or equal to $(k_2,l_2)$ replaced with $\mathcal L((k_2,l_2))$.
\end{enumerate}

These four possibilities can be illustrated through our running example as follows.
\begin{enumerate}[label=\arabic*.]
\item $p=(x_4,x_5)$, $(k_1,l_1)=x_4$, $(k_2,l_2)=x_5$, $p'=(x_4)$.
\item $p=(x_2,x_4,x_5)$, $(k_1,l_1)=x_2$, $(k_2,l_2)=x_4$, $p'=(x_4,x_5)$.
\item $p=(x_1,x_3,x_5)$, $(k_1,l_1)=x_1$, $(k_2,l_2)=x_3$, $p'=(x_1,x_2,x_4)$.
\item $p=(x_1,x_3,x_5)$, $(k_1,l_1)=x_3$, $(k_2,l_2)=x_5$, $p'=(x_4,x_5)$. \qedhere
\end{enumerate}
\end{proof}

Before we proceed to the case of a singular $\lambda$ observe that whether a tuple of antichains $$(A_i\subset Q_i,1\le i\le n-1)$$ provides a vertex of $P_\lambda$ only depends on the tuple itself and not on $\lambda$ (as long as $\lambda$ is regular). If such a tuple does indeed provide a vertex we will refer to it as ``nice''.\label{nice}

Now let $\lambda$ be singular. 
\begin{theorem}\label{allsing}
A point $v$ constitutes a vertex of $P_\lambda$ if and only if it can be expressed as $$v=\sum_{i=1}^{n-1} a_i\chi_{A_i}$$ for a nice tuple $(A_i)$. 
\end{theorem}
\begin{proof}
Once again, to prove the ``if'' part we are to show that for a point $x$ satisfying
\begin{enumerate}[label=\arabic*)]
\item $x_{k,l}=0$ whenever $v_{k,l}=0$ and 
\item $S(x,d)=S(v,d)$ for every Dyck path $d$ such that $S(v,d)=M(\lambda,d)$,
\end{enumerate}
we have $x=v$.

However, we see that if $(k,l)\not\in\bigcup A_i$, then $v_{k,l}=0$ and that if a path $$d=((i_1,j_1),\ldots,(i_N,j_N))$$ intersects every $A_i$ with $i\in[i_1,i_N]$ exactly once, then $S(v,d)=M(\lambda,d)$. This means that we may apply the above proof of the ``if'' part of Theorem~\ref{allequiv}.

To obtain the ``only if'' part we make use of the notions of polar duality and normal fans. Consider a regular dominant integral weight $\lambda'$. Let $(A_i)$ be a nice tuple corresponding to vertices $v$ of $P_\lambda$ and $v'$ of $P_{\lambda'}$. As we have pointed out, we have $v_{k,l}=0$ whenever $v'_{k,l}=0$ and $S(v,d)=M(\lambda,d)$ whenever $S(v',d)=M(\lambda',d)$ for a Dyck path $d$. This shows that the (translated) tangent cone $$C_v=\{k(x-v),x\in P_\lambda,k\ge 0\}$$ is contained in the tangent cone $$C_{v'}=\{k(x'-v'),x'\in P_{\lambda'},k\ge 0\}.$$ We obtain the reverse inclusion of polar duals: $C_{v'}^\circ\subset C_v^\circ$, where $^\circ$ denotes the polar dual with respect to an arbitrary scalar product: $$X^\circ=\{y|\forall x\in X: (x,y)\le 1\}.$$

However, we know that the union of such $C_{v'}^\circ$ over all vertices $v'$ of $P_{\lambda'}$ provided by nice tuples is the whole space $\mathbb R^{n(n-1)/2}$ (they are the maximal cones in the normal fan of $P_{\lambda'}$). Consequently, in view of the established inclusion, the union of all $C_v^\circ$ over $v$ provided by nice tuples (maximal cones in the normal fan of $P_\lambda$) is also $\mathbb R^{n(n-1)/2}$. This shows that $P_\lambda$ does not contain any other vertices.

The details regarding polar duality and normal fans can be found in a textbook on toric varieties such as~\cite{fult}.
\end{proof}

For completeness' sake we point out the obvious fact that two nice tuples $(A_i)$ and $(A'_i)$ provide the same vertex if and only if we have $A_i=A'_i$ for all $i$ such that $a_i>0$.

\section{Permutation vertices}\label{permsec}
The goal of this section is to prove Theorem \ref{permverts}. Throughout this and the next section for a Dyck path $d$ we will denote $M(d)=M(\lambda,d)$.

Recall the set $R$ of all positive length segments $[i,j]$ with $1\le i<j\le n$. Also recall the points $x(E)\in U$ for $E\subset R$. We denote by ${\bf R}_p$ 
($p$ is for permutation) the set of all $E\subset R$ such that if $[i,j]\in E$, $[k,l]\in E$, then the intersection $[i,j]\cap [k,l]$ (if nonempty) belongs to $E$.

We first introduce the following notions. For a Dyck path $d$ we say that $(i,j)\in d$ is a peak of $d$ if $(i,j-1), (i+1,j)\in d$. We say that $(i,j)\in d$ is a valley of $d$ if $(i-1,j), (i,j+1)\in d$. We will make repeated use of 
\begin{lemma}\label{super}
Let $E\in{\bf R}_p$ and $d$ be a Dyck path with peak $(i,j)$. Then $$S(x(E),d)-x(E)_{i,j}\le M(d).$$ If $\lambda$ is regular, then the inequality is strict.
\end{lemma}
\begin{proof}
Denote $d=((i_1,j_1),\ldots,(i_N,j_N))$ and $(i,j)=(i_k,j_k)$. We prove the lemma by induction on $i_N-i_1$. The base is trivial. 

Let $i_N>i_1$. First of all, consider any path $d'$ covered by the induction hypothesis. If $d'$ has a peak $(i',j')$ with $[i',j']\notin E$, then $$S(x(E),d')=S(x(E),d')-x(E)_{i',j'}\le M(d').$$ If $[i',j']\in E$ for any peak $(i',j')$ of $d'$, then $[i',j']\in E$ for any valley $(i',j')$ of $d'$ as well, since $E\in{\bf R}_p$. We write
\begin{multline}\label{pvdecomp}
S(x(E),d')=\sum_{\text{peak }(i',j')\text{ of }d}S(x(E),d^{i',j'})-\sum_{\text{valley }(i',j')\text{ of }d}S(x(E),d^{i',j'})=\\ \sum_{\text{peak }(i',j')\text{ of }d}M(d^{i',j'})-\sum_{\text{valley }(i',j')\text{ of }d}M(d^{i',j'})=M(d').
\end{multline}
In either case we obtain
\begin{equation}\label{hypoth}
S(x(E),d')\le M(d').
\end{equation}

Now consider several possibilities.
\begin{enumerate}[label=\arabic*.]
\item We have $[i_l,j_l]\notin E$ (and $x(E)_{i_l,j_l}=0$) for any $1\le l<k$. Then consider the Dyck path $d'$ obtained from $d$ by replacing $(i_1,j_1),\ldots,(i_k,j_k)$ with $(i_k+1,i_k+2),\ldots,(i_k+1,j_k-1)$. Note that any $x(E)_{i',j'}$ with $(i',j')\in d'\backslash d$ is nonnegative, since either $[i',j']\notin E$ or the induction hypothesis for $d^{i',j'}$ gives us $x(E)_{i',j'}\ge S(d^{i',j'})-M(d^{i',j'})=0$. Consequently, applying~(\ref{hypoth}) to $d'$, we obtain $$S(x(E),d)-x(E)_{i,j}\le S(x(E),d')\le M(d')\le M(d).$$ If $\lambda$ is regular, then $M(d')<M(d)$.
\item We have $[i_l,j_l]\notin E$ for any $k<l\le N$. This case is symmetric to the previous one.
\end{enumerate} 
If neither of the above holds, then we can specify the greatest $l_1<k$ and the least $l_2>k$ such that $[i_{l_1},j_{l_1}]\in E$ and $[i_{l_2},j_{l_2}]\in E$.
\begin{enumerate}[label=\arabic*.]
\setcounter{enumi}{2}
\item We have $j_{l_1}<i_{l_2}$. Then we define Dyck paths $d_1$ obtained from $d$ by replacing all elements succeeding $(i_{l_1},j_{l_1})$ with $(i_{l_1}+1,j_{l_1}),\ldots,(j_{l_1}-1,j_{l_1})$ and $d_2$ obtained from $d$ by replacing all elements preceding $(i_{l_2},j_{l_2})$ with $(i_{l_2},i_{l_2}+1),\ldots,(i_{l_2},j_{l_2}-1)$. Similarly to case 1, any $x(E)_{i',j'}$ with $(i',j')\in (d_1\cup d_2)\backslash d$ is nonnegative. Via~(\ref{hypoth}) applied to $d_1$ and $d_2$ we have $$S(x(E),d)-x(E)_{i,j}\le S(x(E),d_1)+S(x(E),d_2)\le M(d_1)+M(d_2)\le M(d).$$ If $\lambda$ is regular, then $M(d_1)+M(d_2)< M(d)$.
\item We have $j_{l_1}>i_{l_2}$. Then $[i_{l_2},j_{l_1}]\in E$. Define $d_1$ and $d_2$ as in the previous case. We have 
\begin{multline*}
S(x(E),d)-x(E)_{i,j}\le S(x(E),d_1)+S(x(E),d_2)-S(x(E),d^{i_{l_2},j_{l_1}}) \le\\ M(d_1)+M(d_2)-M(d^{i_{l_2},j_{l_1}})=M(d).
\end{multline*}
If $\lambda$ is regular, then the induction hypothesis for $d^{i_{l_2},j_{l_1}}$ provides $x(E)_{i_{l_2},j_{l_1}}>0$ which implies that the first inequality in the above chain is strict. \qedhere
\end{enumerate}   
\end{proof}

\begin{lemma}
Let $E\in {\bf R}_p$. Then $x(E)$ is a vertex of $P_\lambda$.
\end{lemma}
\begin{proof}
We first show that $x(E)\in P_\lambda$. We need to prove that for any $1\le i<j\le n$ the entry $x(E)_{i,j}$ is nonnegative and that for any Dyck path $d$ one has
$S(x(E),d)\le M(d)$. 

If $[i,j]\notin E$, then $x(E)_{i,j}=0$. Let $[i,j]\in E$ and $S(x(E),d^{i,j})=M(d^{i,j})$. Then $x(E)_{i,j}\ge 0$ follows directly from Lemma~\ref{super} applied to $d^{i,j}$ and peak $(i,j)$.

If $d$ is a Dyck path, then we may obtain $S(x(E),d)\le M(d)$ just as we obtained~\eqref{hypoth} (with $d$ replacing $d'$).

The last thing to prove is that $x(E)$ is a vertex of $P_\lambda$, but this now follows directly from the definition of $x(E)$.  
\end{proof}

\begin{lemma}\label{injection}
There exists an injection $\Psi:{\bf R}_p\to S_n$, $\Psi:E\mapsto w(E)$ such that for any $E\in {\bf R}_p$ we have $\mu_{x(E)}=w(E)\lambda$.
\end{lemma}
\begin{proof}
Given $E\in {\bf R}_p$, we define $w(E)$ as follows: $w(E)$ is the product of all the transpositions $(i,j)$ such
that $[i,j]\in E$; the order of factors is chosen in such a way that if $[a,b], [c,d]\in E$ and $b-a<d-c$, then
the transposition $(a,b)$ is on the right of $(c,d)$ (i.e. $(a,b) $ is applied first).

Consider an arbitrary $\mathfrak{sl}_n$-weight $\nu$ and  $E\in{\bf R}_p$. We define the point $x_\nu(E)\in U$ via $x_\nu(E)_{i,j}=0$ if $[i,j]\notin E$ and $S(x_\nu(E),d^{i,j})=M(\nu,d^{i,j})$ if $[i,j]\in E$ (in particular $x_\lambda(E)=x(E)$). We also set $$\mu_\nu(E)=\nu-\sum_{1\le i<j\le n} x_\nu(E)_{i,j}\alpha_{i,j}$$ (in particular $\mu_\lambda(E)=\mu_{x(E)}$).

We see that for any given $E$ the weight $\mu_\nu(E)$ depends linearly on $\nu$. Therefore, to prove the identity $\mu_\lambda(E)=w(E)\lambda$ it suffices to prove $\mu_\nu(E)=w(E)\nu$ for $\nu$ ranging over some spanning set in $\mathfrak h^*$. The set we choose consists of $\varepsilon_i=\omega_i-\omega_{i-1}$ for $1\le i\le n$ where we set $\omega_0=\omega_n=0$. This set is convenient since for any $w\in S_n$ and $1\le i\le n$ we have $w\varepsilon_i=\varepsilon_{w(i)}$. For $1\le i<j\le n$ we also have $\alpha_{i,j}=\varepsilon_i-\varepsilon_j$. (These facts concerning root systems of type $A$ may be found in~\cite{carter}.)

We fix a $E\in{\bf R}_p$ and $1\le i\le n$ and show that $\mu_{\varepsilon_i}(E)=w(E)\varepsilon_i=\varepsilon_{w(E)(i)}$. To do so we define a subset $E(i)\subset E$, $$E(i)=\{[i_1,j_1],\ldots,[i_m,j_m]\}.$$ If the are no segments in $E$ with an endpoint in $i$ we set $E(i)=\varnothing$. Otherwise, let $[i_1,j_1]$ be the shortest segment in $E$ with an endpoint in $i$. Note that, since $E\in{\bf R}_p$, the point $i$ is either the left end for all segments in $E$ containing $i$ as an endpoint or the right end for all of them. Assume $i_1=i$, then we define $[i_2,j_2]$ as the shortest segment in $E$ with a (necessarily right) endpoint in $j_1$ and longer than $[i_1,j_1]$. We we define $[i_3,j_3]$ as the shortest segment with a (left) endpoint in $i_2$ and longer than $[i_2,j_2]$ and so forth while possible. The case $j_1=i$ is symmetrical.

First of all, directly from the definition of $w(E)$ it follows that $w(E)(i)=p$, where $p$ is that endpoint of $[i_m,j_m]$ which is not contained in $[i_{m-1},j_{m-1}]$. Next, observe that $M(\varepsilon_i,d^{k,l})$ is 1 if $k=i$, $-1$ if $l=i$ and 0 otherwise. Herefrom it is easy to obtain a description of the coordinates of $x_{\varepsilon_i}(E)$. Namely, $x_{\varepsilon_i}(E)_{k,l}=0$ if $[k,l]\notin E(i)$, $x_{\varepsilon_i}(E)_{i_1,j_1}$ is 1 if $i_1=i$ and $-1$ if $j_1=i$ and, lastly, $x_{\varepsilon_i}(E)_{i_r,j_r}=-x_{\varepsilon_i}(E)_{i_{r-1},j_{r-1}}$ for $2\le r\le m$. A simple calculation now shows that that $\mu_{\varepsilon_i}(E)=\varepsilon_p$, quod erat demonstrandum.

 
We are left to show that $\Psi$ is an injection. Since $\Psi$ is independent of $\lambda$ we may assume that $\lambda$ is regular. However, in that case for two distinct elements $E,E'\in {\bf R}_p$ we have $x(E)\neq x(E')$ and, consequently, $\mu_{x(E)}\neq\mu_{x(E')}$. The fact that $x(E)\neq x(E')$ can be deduced from Lemma~\ref{lemma} in the next section (where $\lambda$ is assumed to be regular), for the lemma shows that the sets of nonzero coordinates of $x(E)$ and $x(E')$ differ.
\end{proof}

\begin{lemma}\label{bijection}
The map $\Psi:E\mapsto w(E)$ is a bijection between ${\bf R}_p$ and $S_n$.
\end{lemma}
\begin{proof}
We already know that $\Psi$ is an injection. It is thus suffices to construct an inverse map
from $S_n$ to ${\bf R}_p$. 

For $w\in S_n$ we define $\Psi^{-1}(w)$ as the output of a certain algorithm. The algorithm starts out with the identity permutation $u$ and the empty set $E$ and then processes all segments $[i,j]\in R$ in order of non-decreasing length updating the values of $u$ and $E$ each time. A segment $[i,j]$ is processed as follows. If $u^{-1}(i)<w^{-1}(i)$ and $u^{-1}(j)>w^{-1}(j)$ we add $[i,j]$ to $E$ and replace $u$ with $(i,j)u$; otherwise we do not change the values of $u$ and $E$. We will show that after all segments have been processed we end up with $u=w$ and some $E\in{\bf R}_p$ which turns out to be precisely $\Psi^{-1}(w)$. In fact, it suffices to show that our algorithm outputs $u=w$ and $E\in{\bf R}_p$, the fact that $\Psi(E)=w$ will then follow.

During the course of our algorithm for an $i$ with $w^{-1}(i)>i$ the value of $u^{-1}(i)$ gradually increases until it reaches $w^{-1}(i)$. Similarly, for a $j$ with $w^{-1}(j)<j$ the value of $u^{-1}(j)$ gradually decreases until it reaches $w^{-1}(j)$. More specifically, consider the step of our algorithm dealing with some segment $[i,j]$. Call all $k\in[1,n]$ such that we have $u^{-1}(k)<w^{-1}(k)$ at beginning of our step ``white''. Similarly, call all $k\in[1,n]$ such that we have $u^{-1}(k)>w^{-1}(k)$ at beginning of our step ``black''. It is straightforward to inductively verify the following properties of our algorithm.
\begin{enumerate}[label=\arabic*.]
\item For any pair $k<l$ of white numbers we have $u^{-1}(k)<u^{-1}(l)$. The same goes for any pair of black numbers. 
\item For a white $k$ and a black $l$ with $k<l$ we have $u^{-1}(k)>u^{-1}(l)$ if and only if the segment $[k,l]$ has already been processed at a preceding step.
\item If $i$ is white and $j$ is black, then all $k$ with $u^{-1}(i)<u^{-1}(k)<u^{-1}(j)$ are neither white nor black. We add $[i,j]$ to $E$ as a result of our step if and only if $i$ is white and $j$ is black.
\item The set of numbers which will be white during the next step is either the same as for the current step or that set minus $i$. The set of numbers which will be black during the next step is either the same as for the current step or that set minus $j$. 
\end{enumerate}

From property 2 we deduce that after the last step for any $i$ and $j$ such that $u^{-1}(i)<w^{-1}(i)$ and $u^{-1}(j)>w^{-1}(j)$ we must have $u^{-1}(j)<u^{-1}(i)$. This is only possible when $u=w$. Now suppose that after the last step we have $[i_1,j_1],[i_2,j_2]\in E$ for some $i_1<i_2<j_1<j_2$. Then, due to property 3, $i_2$ was white during the step dealing with $[i_2,j_2]$ and $j_1$ was black during the step dealing with $[i_1,j_1]$. Property 4 then implies that $i_2$ was white and $j_1$ was black during the step dealing with $[i_2,j_1]$ and we must have $[i_2,j_1]\in E$ via property 3. Therefore, we have $E\in{\bf R}_p$ after the last step.
\end{proof}

\begin{cor}
Any permutation in $w\in S_n$ can be uniquely (up to transpositions of consecutive commuting factors) written as a product of transpositions $w=(l_1,m_1)\dots (l_N,m_N)$ 
($l_i<m_i$ for all $i$) such that $m_i-l_i\ge m_{i+1}-l_{i+1}$ and the set $\{[l_i,m_i]\}_{i=1}^N$ belongs to ${\bf R}_p$. 
\end{cor}

We close this section with the discussion on the extremal part of the PBW character. 
We are interested on the following statistics on the set ${\bf R}_p$: 
\[
b(E)=\sum_{1\le i<j\le n} x(E)_{i,j}.
\]
Let us consider the PBW degenerate representation $V^a_\lambda$ of $\mathfrak{sl}_n$ (see \cite{FFL1}). 
This representation has an additional PBW grading; in particular, each extremal weight subspace of 
$V_\lambda^a$ being one-dimensional has some PBW degree. Then the PBW degree of the weight 
$w(E)\lambda$ subspace is equal to $b(E)$ (the extremal PBW degrees were computed in \cite{CF}).  

\begin{example}
Let $n=3$, $\la=a_1\omega_1+a_2\omega_2$. Then $\sum_{E\in{\bf R}_p} q^{b(E)}=1+q^{a_1}+q^{a_2}+3q^{a_1+a_2}$.
\end{example}

Let $\la_i=a_i+\dots+a_{n-1}$, $i=1,\dots,n-1$ and $\la_n=0$. Then
\begin{proposition}
$2b(E)=\sum_{i=1}^n |\la_i-\la_{w(E)i}|$.
\end{proposition}
\begin{proof}
The right- and left-hand sides are linear in $\lambda$. Hence it suffices to prove the claim for 
fundamental weights, which is clear.  
\end{proof}

\section{Simple vertices}\label{simpsec}
The goal of this section is to prove Theorems \ref{simpleverts} and \ref{simpleperm}. 
We assume that $\la$ is regular.

The  defining inequalities of the polytope $P_\la$ are of two sorts: either $x_{i,j}\ge 0$ or 
$S(x,d)\le M(d)$. Therefore, a simple vertex $x$ is defined by $n(n-1)/2$ equalities of two sorts:
either a coordinate $x_{i,j}$ vanishes or the value of $S(x,d)$ is maximal possible for a Dyck path $d$.
For a simple vertex $x$ we denote by $Z(x)\subset \{(i,j): 1\le i<j\le n\}$ the set consisting of pairs
$(i,j)$ such that $x_{i,j}=0$. We also denote by $D(x)$ the set
of Dyck paths $d$ such that $S(x,d)=M(d)$ (in particular, $|Z(x)|+|D(x)|=n(n-1)/2$).

We first prove the following lemma.
\begin{lemma}\label{nv}
Let $x$ be a simple vertex of $P_\la$ and let $d\in D(x)$. Then for any peak or a valley $(i,j)$ of $d$
the coordinate $x_{i,j}$ does not vanish.
\end{lemma}
\begin{proof}
We prove the claim for a peak, the proof for a valley is very similar. 
Assume that $x_{i,j}=0$
for a peak $(i,j)$ of $d\in D(x)$. 

First, let $j-i>2$. Then let $d'=d\setminus (i,j)\cup (i+1,j-1)$
be a new Dyck path. Since $x_{i,j}=0$, one has  
\[
M(d')\ge S(x,d')\ge S(x,d)=M(d).
\]
Now the equality $M(d')=M(d)$ implies $S(x,d')=S(x,d)$ and thus $x_{i+1,j-1}=x_{i,j}=0$. 
But the left-hand sides of the four equalities $S(x,d)=M(d)$, $S(x,d')=M(d')$, $x_{i,j}=0$ and $x_{i+1,j-1}=0$
are linearly dependent, which contradicts the simplicity of vertex $x$. 

Now let $j-i=2$. Let  $d',d''\subset d$ be two Dyck paths (not containing $(i,j)$) 
defined by the condition $d=d'\cup (i,j)\cup d''$. Then since $x_{i,j}=0$ and $j-i=2$, we have
$S(x,d')=M(d')$, $x_{i,j}=0$, $S(x,d'')=M(d'')$ and $S(x,d)=M(d)$. Since $M(d)=M(d')+M(d'')$ and
the linear forms $x_{i,j}$, $S(x,d')$ and $S(x,d'')$ sum up to $S(x,d)$, we obtain a contradiction again.   
\end{proof}

Now let us prove that any simple vertex of $P_\la$ is of the form $x(E)$.
\begin{proposition}\label{simpE}
Let $x$ be a simple vertex of the polytope $P_\la$. Then there exists $E\subset R$ such that 
$x=x(E)$.
\end{proposition}
\begin{proof}
Let $d\in D(x)$ and let $P(d)$, $V(d)$ be the sets of peaks and valleys of $d$. Then one has
\begin{align*}
S(x,d) & =\sum_{(i,j)\in P(d)} S(x,d^{i,j}) - \sum_{(i,j)\in V(d)} S(x,d^{i,j}),\\
M(d) & =\sum_{(i,j)\in P(d)} M(d^{i,j}) - \sum_{(i,j)\in V(d)} M(d^{i,j}),
\end{align*}
and hence
\begin{equation}\label{PV}
\sum_{(i,j)\in P(d)} S(x,d^{i,j}) - \sum_{(i,j)\in V(d)} S(x,d^{i,j})=
\sum_{(i,j)\in P(d)} M(d^{i,j}) - \sum_{(i,j)\in V(d)} M(d^{i,j}).
\end{equation}
Lemma \ref{nv} tells us that for all peaks and valleys $(i,j)$ of Dyck paths $d\in D(x)$ the coordinates
$x_{i,j}$ are strictly positive. Therefore,~\eqref{PV} tell us that the space generated by the linear forms $S(\cdot,d)$, $d\in D(x)$ is contained in the space generated by the linear forms $S(\cdot,d^{i,j})$, $x_{i,j}>0$. Both of these sets of forms are linearly independent and are equal in size, hence the linear spans coincide. Consequently, for each pair $(i,j)$ with $x_{i,j}>0$ there exists a linear combination of the left-hand sides of \eqref{PV} giving $S(x,d^{i,j})$. Obviously, the right-hand side of such a linear combination is equal to $M(d^{i,j})$. We conclude that $x=x(E)$, where $E$ is the set of all $[i,j]$ such that $x_{i,j}> 0$.  
\end{proof}

To proceed we will require
\begin{lemma}\label{lemma}
Let $E\in{\bf R}_p$. Then for any $[i,j]\in E$ we have $x(E)_{i,j}>0$. Moreover, let Dyck path $d$ be such that $S(x(E),d)=M(d)$. Then we have $[i,j]\in E$ whenever $(i,j)$ is a peak or a valley of $d$.
\end{lemma}
\begin{proof}
Since $\lambda$ is regular, the first statement follows from Lemma~\ref{super} applied to $d^{i,j}$ and peak $(i,j)$. 

If $d$ has a peak $(i,j)$ with $[i,j]\notin E$, then Lemma~\ref{super} applied to $d$ and peak $(i,j)$ implies that $S(x(E),d)<M(d)$ which contradicts our assumption on $d$.
\end{proof}

Now let ${\bf R}_s$ be the set of $E\subset R$ such that if two segments of $E$ intersect, then one is contained inside the other.
In particular, ${\bf R}_s\subset {\bf R}_p$.  
\begin{theorem}\label{simpequiv}
For $E\subset R$ the point $x(E)$ is a simple vertex of $P_\lambda$ if and only if $E\in {\bf R}_s$.
\end{theorem}
\begin{proof}
Suppose that $x(E)$ is a simple vertex but we have two segments $[i,j],[k,l]\in E$ with a nonempty intersection different from both of the segments. Assuming that $i<k$, this intersection is $[k,j]$. Let $d$ be a Dyck path defined by
\[
d=((d^{i,j}\cup d^{k,l})\setminus d^{k,j})\cup (k,j).
\]
Then $d$ has peaks $(i,j)$ and $(k,l)$ and only one valley $(k,j)$. Hence
\[
S(x(E),d)+S(x(E),d^{k,j})=S(x(E),d^{i,j})+S(x(E),d^{k,l}).
\]
Since $M(d)+M(d^{k,j})=M(d^{i,j})+M(d^{k,l})$ and $S(x(E),d^{i,j})=M(d^{i,j})$, $S(x(E),d^{k,l})=M(d^{k,l})$ but $S(x(E),d)\le M(d)$, $S(x(E),d^{k,j})\le M(d^{k,j})$,
we conclude that 
\[
S(x(E),d)= M(d),\ S(x(E),d^{k,j})=M(d^{k,j}).
\]
But this contradicts the simplicity of $x(E)$.

Now let $E$ be a set in ${\bf R}_s$. We prove that $x(E)$ is a simple vertex. To do so it suffices to verify two facts. First, that for $[i,j]\in E$ we have $x(E)_{i,j}\neq 0$ and, second, that for a Dyck path $d$ different from any $d^{i,j}$ with $[i,j]\in E$ we have $S(x(E),d)<M(d)$. The first fact is immediate from the first part of Lemma~\ref{lemma} and the second fact follows from the second part of Lemma~\ref{lemma} in view of the definition of ${\bf R}_s$.
\end{proof}

Since ${\bf R}_s\subset {\bf R}_p$, all the simple vertices are permutation vertices. Therefore it is natural to
ask when a vertex $T_w$ is simple.  
\begin{theorem}
A vertex $T_w$ is simple if and only if for any $1\le i<j\le n$ the condition $w^{-1}(j)\le i$ implies $w^{-1}(i+1)\le j$. 
The number of simple vertices is equal to the $(n-1)$-st (large) Schr\"oder number.
\end{theorem}
\begin{proof}
We first show that $|{\bf R}_s|$ is equal to the $(n-1)$-st Schr\"oder number. Indeed, denote $|{\bf R}_s|=X_{n-1}$. Then we see that the number of $E\in{\bf R}_s$, such that $[1,j]\notin E$ for all $j$,
is $X_{n-2}$ while the number of $E\in{\bf R}_s$ containing the segment $[1,n]$ is $X_{n-1}/2$. Furthermore, for a $2\le k\le n-1$ the number of $E\in{\bf R}_s$ such that $[1,k]$ is the longest segment of the form $[1,j]$ contained in $E$ is equal to $X_{k-1}/2\cdot X_{n-k-1}$. We obtain the recurrence relation $$X_{n-1}=X_{n-2}+X_{n-1}/2+\sum_{k=2}^{n-1}X_{k-1}/2\cdot X_{n-k-1} \Longleftrightarrow X_{n-1}=X_{n-2}+\sum_{k=0}^{n-2} X_k X_{n-k-2}.$$ This, however, coincides with the recurrence relation for Schr\"oder numbers.

Next, recall the algorithm from the proof of  Lemma~\ref{bijection} via which we described the map $\Psi^{-1}: S_n\to {\bf R}_p$. We prove the ``only if'' part by showing that if for a pair $i<j$ we have $w^{-1}(j)\le i$ and $w^{-1}(i+1)>j$, then $\Psi^{-1}(w)\not\in {\bf R}_s$.

Indeed, we see that $w^{-1}(j)<j$ and $w^{-1}(i+1)>i+1$. Consequently, in terms of the proof of Lemma~\ref{bijection}, the number $j$ was black during the first several steps of our algorithm. Consider the longest segment of form $[k,j]$ in $\Psi^{-1}(w)$. During the step dealing with segment $[k,j]$ we must have had $u^{-1}(k)=w^{-1}(j)$ and $k$ must have been white, i.e. we must have had $k\le w^{-1}(j)$. Similarly, $i+1$ was white during the first several steps of our algorithm, consider the longest segment in $\Psi^{-1}(w)$ of the form $[i+1,l]$ and deduce that $l\ge w^{-1}(i+1)$. We see that the segments $[k,j]$ and $[i+1,l]$ intersect non-trivially and prevent $\Psi^{-1}(w)$ from lying in ${\bf R}_s$.

Finally, in \cite{CFR2} it was proved that the number of $w$ such that for any $1\le i<j\le n$ the condition $w^{-1}(j)\le i$ implies $w^{-1}(i+1)\le j$ is the $(n-1)$-st Schr\"oder number. The theorem follows.
\end{proof}

\section{Applications to type $C$}\label{Csec}

The above study of the vertices of FFLV polytopes of type $A$ can be applied to obtain similar results for type $C$. Let us first define FFLV bases and FFLV polytopes of type $C$ (see \cite{FFL2}).

Consider the symplectic Lie algebra $\mathfrak{sp}_{2n}$ with Cartan subalgebra $\mathfrak h$ and simple roots $\alpha_1,\ldots,\alpha_n\in\mathfrak h^*$ such that $(\alpha_i,\alpha_{i+1})<0$ and $\alpha_n$ is a long root. We fix a integral dominant weight $\lambda\in\mathfrak h^*$ with (nonnegative integer) coordinates $(a_1,\ldots,a_n)$ in the corresponding basis of fundamental weights. We denote the corresponding irreducible representation $L_\lambda$. (Of course, these and some of the below notation conflicts with that introduced above. Throughout this final section we will give preference to the new notation referring to $\mathfrak{sp}_{2n}$ disregarding the old one.)

Once again, the FFLV basis in $L_\lambda$ is a monomial basis parametrized by certain number triangles.  Each such triangle $T$ consists of $n^2$ numbers $T_{i,j}$ with $1\le i<j \le 2n+1-i$. We again visualize these triangles with $T_{i,j}$ and $T_{i+1,j+1}$ being, respectively, the upper-left and the upper-right neighbors of $T_{i,j+1}$, e.g. for $n=3$ we have:
\begin{center}
\begin{tabular}{ccccc}
$T_{1,2}$ && $T_{2,3}$ && $T_{3,4}$\\
&$T_{1,3}$ && $T_{2,4}$ & \\
&& $T_{1,4}$ && $T_{2,5}$ \\
&&& $T_{1,5}$ &\\
&&&& $T_{1,6}$
\end{tabular}
\end{center}
\vspace{2ex}

To define the set $\Pi_\lambda$ of triangles parametrizing the FFLV basis we use a slightly modified notion of a Dyck path. We call a sequence $$d=((i_1,j_1),\ldots,(i_N,j_N))$$ of pairs $1\le i<j \le 2n+1-i$ a {\it type $C$ Dyck path} if we have $j_1-i_1=1$, for all $1\le k\le N-1$ the element $(i_{k+1},j_{k+1})$ is either $(i_k+1,j_k)$ or $(i_k,j_k+1)$ and, lastly, we either have $j_N-i_N=1$ or $i_N+j_N=2n+1$. The difference is that now the path must not necessarily end in the top row but may as well end in the rightmost vertical column.

For a triangle $T=(T_{i,j}, 1\le i<j \le 2n+1-i)$ and type $C$ Dyck path $$d=((i_1,j_1),\ldots,(i_N,j_N))$$ we denote $$S(T,d)=\sum_{(i,j)\in d} T_{i,j}.$$ Next, if we have $j_N-i_N=1$ we define $$M(\lambda,d)=a_{i_1}+\ldots+a_{i_N},$$ otherwise (i.e. we have $i_N+j_N=2n+1$) we define $$M(\lambda,d)=a_{i_1}+\ldots+a_n.$$

We then define $\Pi_\lambda$ by saying that $T\in\Pi_\lambda$ if and only if all $T_{i,j}$ are nonnegative integers and for any type $C$ Dyck path $d$ we have $S(T,d)\le M(\lambda,d)$. 

Now for $1\le i<j \le 2n+1-i$ let $f_{i,j}\in\mathfrak{sp}_{2n}$ be a nonzero element in the root space of the negative root $$-\alpha_{i,j}=-\alpha_i-\ldots-\alpha_{j-1}$$ where for $k>n$ we set $\alpha_k=\alpha_{2n-k}$. We choose a highest weight vector $v_0\in L_\lambda$ and define the basis vector corresponding to $T\in\Pi_\lambda$ as $$v_T=\left(\prod_{i,j} f_{i,j}^{T_{i,j}}\right) v_0,$$ where an arbitrary order of the factors $f_{i,j}^{T_{i,j}}$ is fixed. The fact that such a set of vectors does indeed constitute a basis in $L_\lambda$ was established in~\cite{FFL2}. 

We proceed to define the type C FFLV polytope $P_\lambda$. Consider the space $U=\mathbb R^{n^2}$ with coordinates labeled by pairs $1\le i<j \le 2n+1-i$. Then $P_\lambda\subset U$ consists of points $x=(x_{i,j})$ such that all $x_{i,j}$ are nonnegative and for any type $C$ Dyck path $d$ we have $S(x,d)\le M(\lambda,d)$. We see that $P_\lambda$ is a convex polytope the set of integer points in which is precisely $\Pi_\lambda$. 

In order to employ the results obtained for type $A$ we also consider the algebra $\mathfrak{sl}_{2n}$ with a fixed Cartan decomposition and set of simple roots. Then for any integral dominant $\mathfrak{sl}_{2n}$-weight $\lambda_A$ we may view the corresponding type $A$ FFLV polytope $P_{\lambda_A}$ as a subset of the space $U_A=\mathbb R^{n(2n-1)}$ with coordinates labeled by pairs $1\le i<j\le 2n$. We will describe sets of vertices of $P_\lambda$ by considering various embeddings $\varphi:U\hookrightarrow U_A$ and establishing relationships between $\varphi(P_\lambda)$ and certain $P_{\lambda_A}$.

Our fist result concerning $P_\lambda$ describes the set of all of its vertices. We have the trivial embedding $\iota:U\hookrightarrow U_A$ such that for pairs $1\le i<j \le 2n+1-i$ we have $\iota(x)_{i,j}=x_{i,j}$ and the remaining $n(n-1)$ coordinates of $\iota(x)$ are zero. We will identify $U$ with $\iota(U)$, i.e. view $U$ as a subspace in $U_A$.

Consider the $\mathfrak{sl}_{2n}$-weight $\lambda_0$ with coordinates $$(a_1,\ldots,a_n,0,\ldots,0)$$ in the basis of fundamental weights. We see that $$P_\lambda=P_{\lambda_0}\cap U$$ and that $P_\lambda$ constitutes a face of $P_{\lambda_0}$, since hyperplanes of the form $\{x_{k,l}=0\}$ intersect $P_{\lambda_0}$ in its faces. Therefore, a point constitutes a vertex of $P_{\lambda}$ if and only if it constitutes a vertex of $P_{\lambda_0}$ and lies in $U$. 

Now, similarly to Section~\ref{allsec}, we have the order relation $\preceq$ on the set of pairs $1\le i<j\le 2n$ and the subsets $Q_i$ therein for all $1\le i\le 2n-1$. For each $1\le i\le n$ we define a subset $Q^C_i\subset Q_i$ consisting of those $(a,b)\in Q_i$ for which 
$a+b\le 2n+1$. We also denote the fundamental weights of $\mathfrak{sl}_{2n}$ via $\omega^A_1,\ldots,\omega^A_{2n-1}$.

\begin{theorem}\label{callverts}
A point in $U$ constitutes a vertex of $P_\lambda$ if and only if it has the form $\sum_{i=1}^n a_i\chi_{A_i}$ for a tuple of $\preceq$-antichains $(A_i\subset Q^C_i,1\le i\le n)$ such that for all $1\le i\le n-1$ the point $\chi_{A_i}+\chi_{A_{i+1}}$ constitutes a vertex of $P_{\omega^A_i+\omega^A_{i+1}}$.
\end{theorem}
Before we prove the theorem note that this description is explicit, since Proposition~\ref{tworoots} tells us exactly when $\chi_{A_i}+\chi_{A_{i+1}}$ constitutes a vertex of $P_{\omega^A_i+\omega^A_{i+1}}$.
\begin{proof}[Proof of Theorem~\ref{callverts}]
The ``only if'' part follows from the fact that every vertex of $P_\lambda$ is a vertex of $P_{\lambda_0}$ via Theorems~\ref{allequiv} and~\ref{allsing}.

Let a tuple of antichains $(A_i\subset Q^C_i,1\le i\le n)$ be such that for all $1\le i\le n-1$ the point $\chi_{A_i}+\chi_{A_{i+1}}$ constitutes a vertex of $P_{\omega^A_i+\omega^A_{i+1}}$. To prove the ``if'' part we must show that this tuple may be extended to a nice (as defined on page~\pageref{nice} but with respect to $\mathfrak{sl}_{2n}$) tuple $(A_i,1\le i\le 2n-1)$. However, such an extension may be obtained by setting $A_i=Q_i\cap A_n$ for all $n<i\le 2n-1$.
\end{proof}

Next we describe the ``Weyl group translate'' vertices, similar to the permutation vertices in type $A$. 

Consider the Weyl group $$W\cong S_n\ltimes (\mathbb Z/2)^n$$ of $\mathfrak{sp}_{2n}$. For every $w\in W$ we have a unique $T_w\in\Pi_\lambda$ such that the vector $v_{T_w}$ has weight $w\lambda$. We will show that each such $T_w$ constitutes a vertex of $P_\lambda$ and describe it explicitly.

To do so we introduce a less trivial embedding $\varphi:U\to U_A$ defined as follows. For a point $x\in U$ and a pair $1\le i<j\le 2n$ we set
$$
\varphi(x)_{i,j}=
\begin{cases}
x_{i,j}, &\text{ if }i+j<2n+1,\\
2x_{i,j}, &\text{ if }i+j=2n+1,\\
x_{2n+1-j,2n+1-i}, &\text{ if }i+j>2n+1. 
\end{cases}
$$
One may say that the number triangle $(\varphi(x)_{i,j})$ is the sum of number triangle $(x_{i,j})$ and its reflection across the line $i+j=2n+1$.

To describe the vertices $T_w$ we recall the set $R$ of subsegments in $[1,2n]$ with integer endpoints and positive length and the set ${\bf R}_p\subset 2^R$ of such $E\subset R$ that $[i,j],[k,l]\in E$ and $[i,j]\cap[k,l]\neq\varnothing$ imply $[i,j]\cap[k,l]\in E$. We also define the $\mathfrak{sl}_{2n}$-weight $$\overline\lambda=(a_1,\ldots,a_{n-1},2a_n,a_{n-1},\ldots,a_1).$$ As we have shown in section~\ref{permsec} each $E\in{\bf R}_p$ defines a permutation vertex of $P_{\overline\lambda}$, we denote this vertex $x(E)$.

We say that $E\in{\bf R}_p$ is symmetric if it contains $[i,j]$ whenever it contains $[2n+1-j,2n+1-i]$.

\begin{theorem}
Whenever $E\in{\bf R}_p$ is symmetric we have $x(E)\in\varphi(P_\lambda)$ and the point $\varphi^{-1}(x(E))$ constitutes a vertex of $P_\lambda$. Moreover, there exists a bijection $\Xi:E\mapsto w_C(E)$ from the set of symmetric $E\in{\bf R}_p$ to $W$ such that $T_{w_C(E)}=\varphi^{-1}(x(E))$ for any symmetric $E\in{\bf R}_p$.
\end{theorem}
\begin{proof}
The image $\varphi(U)$ consists of such $x$ that $x_{i,j}=x_{2n+1-j,2n+1-i}$ for all $1\le i<j\le 2n$. It is easily verified that $\varphi(P_\lambda)=P_{\overline\lambda}\cap\varphi(U)$. Furthermore, if $E\in{\bf R}_p$ is symmetric, then we obviously have $x(E)\in\varphi(U)$. Since $x(E)$ is a vertex of $P_{\overline\lambda}$, we conclude that $\varphi^{-1}(x(E))$ is indeed a vertex of $P_\lambda$. In particular, Theorem~\ref{callverts} shows that $\varphi^{-1}(x(E))$ is an integer point.

Now let us describe the action of $W$ explicitly. We introduce a basis $(\varepsilon_i)$ in $\mathfrak h^*$  such that $\alpha_i=\varepsilon_i-\varepsilon_{i+1}$ for $1\le i\le n-1$ and $\alpha_n=2\varepsilon_n$. For $\sigma\in S_n$ and $r\in(\mathbb Z/2)^n=\{1,-1\}^n$ consider the element $(\sigma,r)\in S_n\ltimes (\mathbb Z/2)^n=W$. Then $(\sigma,r)\varepsilon_i=r_{\sigma(i)}\varepsilon_{\sigma(i)}$ for all $1\le i\le n$. These facts concerning root systems of type $C$ may be found in~\cite{carter}.

The $i$-th fundamental weight $\omega_i$ is equal to $\varepsilon_1+\ldots+\varepsilon_i$ and, consequently, the coordinates of $\lambda$ with respect to the basis $(\varepsilon_i)$ are $$\lambda_i=a_i+\ldots+a_n.$$

We next define the map $\Xi$ by choosing a symmetric $E\in{\bf R}_p$ and setting $$w_C(E)=\prod_{\substack{[i,j]\in E\\i+j\le 2n+1}} ((i,\overline j),r^{i,j})$$ where the factors on the right are ordered by $j-i$ increasing from right to left and the following notation is used. The number $\overline j$ is equal to $j$ if $j\le n$ and $2n+1-j$ if $j>n$ and $(i,\overline j)\in S_n$ is the corresponding transposition. The element $r^{i,j}\in(\mathbb Z/2)^n$ is identical if $j\le n$. If $j>n$ the element $r^{i,j}$ has all coordinates equal to $1$ except for $r^{i,j}_i=r^{i,j}_{2n+1-j}=-1$.


Now let us recall the set of $\mathfrak{sl}_{2n}$-weights $\varepsilon^A_i=\omega^A_i-\omega^A_{i-1}$ for $1\le i\le 2n$ (where $\omega^A_0=\omega^A_{2n}=0$). For $w_A\in S_{2n}$ we have $w_A\varepsilon^A_i=\varepsilon^A_{w_A(i)}$. In Lemma~\ref{injection} we have also defined a permutation $w_A(E)\in S_{2n}$ for every $E\in{\bf R}_p$. Furthermore, denote the simple $\mathfrak{sl}_{2n}$-roots $\alpha^A_i$, $1\le i\le 2n-1$ and the positive roots $$\alpha^A_{i,j}=\alpha^A_i+\ldots+\alpha^A_{j-1}$$ for $1\le i<j\le 2n$. 


For a symmetric $E\in{\bf R}_p$ let $\mu_C(E)$ be the ($\mathfrak{sp}_{2n}$-)weight of $v_{\varphi^{-1}(x(E))}\in L_\lambda$ and let $\mu_A(E)$ be the ($\mathfrak{sl}_{2n}$-)weight of $v_{x(E)}\in L_{\overline\lambda}$. Define a linear map $\psi$ sending $\mathfrak{sp}_{2n}$-weights to $\mathfrak{sl}_{2n}$-weights by $\psi:\varepsilon_i\mapsto\varepsilon^A_i-\varepsilon^A_{2n+1-i}$ (in particular, $\psi(\lambda)=\overline\lambda$). With the help of
$$
\alpha_{i,j}=
\begin{cases}
\varepsilon_i-\varepsilon_j&\text{if }j\le n,\\
\varepsilon_i+\varepsilon_{\overline j}&\text{if }j>n
\end{cases}
$$
for $1\le i<j\le 2n+1-i$ and $\alpha^A_{i,j}=\varepsilon^A_i-\varepsilon^A_j$ for $1\le i<j\le 2n$, it is straightforward to obtain $\psi(\mu_C(E))=\mu_A(E)$. Furthermore, by comparing the definitions of $w_A$ and $w_C$ it is not hard to verify that $\psi(w_C(E)\lambda)=w_A(E)\overline\lambda$.


Since in view of Lemma~\ref{injection} we have $\mu_A(E)=w_A(E)\overline\lambda$ and, moreover, $\psi$ is injective, we conclude that $\mu_C(E)=w_C(E)\lambda$.

The map $\Xi$ is independent of $\lambda$. Moreover, if $\lambda$ is regular, then so is $\overline\lambda$ and, due to Lemma~\ref{lemma}, for symmetric $E\neq E'$ we have $x(E)\neq x(E')$ and $\varphi^{-1}(x(E))\neq\varphi^{-1}(x(E'))$. Therefore $\Xi$ is injective.

We are left to show that $\Xi$ is a bijection, i.e. the number of symmetric $E\in{\bf R}_p$ is equal to $|W|=n! 2^n$. However, visibly, for each symmetric $E$ the permutation $w_A(E)$ is symmetric in the sense that whenever $i+j=2n+1$ we have $w_A(i)+w_A(j)=2n+1$. The order of the subgroup of symmetric permutations in $S_{2n}$ is seen to be $n! 2^n$. Therefore, it suffices to show that, conversely, $\Psi^{-1}(w_A)$ is symmetric whenever $w_A$ is. This is evident from the description of $\Psi^{-1}$ found in the proof of Lemma~\ref{bijection}.
\end{proof}

Our final result describes the set of simple vertices of $P_\lambda$ in the case of a regular $\lambda$. We define the subset ${\bf R}^C_s\subset{\bf R}_p$ consisting of such symmetric $E$ that any two segments $[i,j],[k,l]\in E$ with $i+j\le 2n+1$ and $k+l\le 2n+1$ either do not intersect or have one containing the other.

\begin{theorem}
If $\lambda$ is regular, then a vertex of $P_\lambda$ is simple if and only if it has the form $\varphi^{-1}(x(E))$ for some $E\in{\bf R}^C_s$.
\end{theorem}
\begin{proof}
For every pair $1\le i<j\le 2n+1-i$ we define a type $C$ Dyck path $d^{i,j}_C$. This path has the form $$((i,i+1),(i,i+2),\ldots,(i,j),(i+1,j),(i+2,j),\ldots,(k,j))$$ where either $k=j-1$ or $k+j=2n+1$. In other words, $d_C^{i,j}$ starts in $(i,i+1)$, goes down and to the right until it reaches $(i,j)$ and then continues up and to the right for as long as it can.

Consider some $E\in{\bf R}_p^C$. It is easy to see that for every $[i,j]\in E$ with $1\le i<j\le 2n+1-i$ we have $S(\varphi^{-1}(x(E)),d^{i,j}_C)=M(\lambda,d^{i,j}_C)$. To show that the vertex $\varphi^{-1}(x(E))$ is simple we are to verify two facts. First, that there are no type $C$ Dyck paths $d$ for which $S(\varphi^{-1}(x(E)),d)=M(\lambda,d)$ other than the $d_C^{i,j}$ with $[i,j]\in E$. Second, that $\varphi^{-1}(x(E))_{i,j}>0$ whenever $1\le i<j\le 2n+1-i$ and $[i,j]\in E$.

Since $\overline\lambda$ is a regular $\mathfrak{sl}_{2n}$-weight we may employ Lemma~\ref{lemma}. We immediately obtain the second of the above statements. To prove the first consider a type $C$ Dyck path $$d=((i_1,j_1),\ldots,(i_N,j_N))$$ with $S(\varphi^{-1}(x(E)),d)=M(\lambda,d)$. If $j_N-i_N=1$, then $d$ may be viewed as type $A$ Dyck path with $S(x(E),d)=M(\overline\lambda,d)$. Lemma~\ref{lemma} then tells us that $[i,j]\in E$ for any peak or valley $(i,j)$ of $d$. Due to the definition of  ${\bf R}^C_s$, this is only possible if $d=d_C^{i,j}$ for some $[i,j]\in E$. If $i_N+j_N=2n+1$ then we define a type $A$ Dyck path $$d'=((i_1,j_1),\ldots,(i_N,j_N),(2n+1-j_{N-1},2n+1-i_{N-1}),\ldots,(2n+1-j_1,2n+1-i_1)).$$ $S(\varphi^{-1}(x(E)),d)=M(\lambda,d)$ implies $S(x(E),d')=M(\overline\lambda,d')$, and we see that $[i,j]\in E$ for any peak or valley $(i,j)$ of $d'$. Again, this is only possible if $d=d_C^{i,j}$ for some $[i,j]\in E$.

We proceed to the ``only if'' part. First we show that if $x$ is a simple vertex of $P_\lambda$, then $x=\varphi^{-1}(x(E))$ for some symmetric $E\subset R$. We outline the argument, the details being filled in rather similarly to the proof of Proposition~\ref{simpE}.

For a type $C$ Dyck path $d$ we say that $(i,j)\in d$ is a peak of $d$ if $(i,j-1), (i+1,j)\in d$ or $(i,j-1)\in d$ and $i+j=2n+1$. We say that $(i,j)\in d$ is a valley of $d$ if $(i-1,j), (i,j+1)\in d$ or $(i-1,j)\in d$ and $i+j=2n+1$.

Now let $x$ be a simple vertex of $P_\lambda$ and let $D(x)$ be the set of such type $C$ Dyck paths $d$ that $S(x,d)=M(\lambda,d)$. Then $|D(x)|+|\{x_{i,j}=0\}|=n^2$. If $d\in D(x)$, then for any peak or valley $(i,j)$ we have $x_{i,j}>0$. For all $d\in D(x)$ rewrite the equality $S(x,d)=M(\lambda,d)$ as 
\begin{multline*}
\sum_{\text{peak }(i,j)\text{ of }d}S(x,d_C^{i,j})-\sum_{\text{valley }(i,j)\text{ of }d}S(x,d_C^{i,j})=\\ \sum_{\text{peak }(i,j)\text{ of }d}M(\lambda,d_C^{i,j})-\sum_{\text{valley }(i,j)\text{ of }d}M(\lambda,d_C^{i,j}).
\end{multline*}
Due to $x$ being simple we obtain $S(x,d_C^{i,j})=M(\lambda,d_C^{i,j})$ for any $x_{i,j}>0$ as a linear combination of the above equalities. Consequently, $x=\varphi^{-1}(x(E))$, where $E$ consists of such $[i,j]$ with $1\le i<j\le 2n$ that either $x_{i,j}>0$ or 
$x_{2n+1-j,2n+1-i}>0$. 

We are left to show that for a symmetric $E\subset R$ the point $\varphi^{-1}(x(E)),d)$ is a simple vertex of $P_\lambda$ only if $E\in{\bf R}^C_s$. This is done rather similarly to the proof of the ``only if'' part of Theorem~\ref{simpequiv}. Namely, suppose that $[i,j],[k,l]\in E$ with $1\le i<k<j<l\le 2n$ and $i+j\le 2n+1$ and $k+l\le 2n+1$. We define the type $C$ Dyck path $$d=((d_C^{i,j}\cup d_C^{k,l})\backslash d_C^{k,j})\cup\{(k,j)\}$$ and obtain the relation $$S(\varphi^{-1}(x(E)),d_C^{i,j})+S(\varphi^{-1}(x(E)),d_C^{k,l})=S(\varphi^{-1}(x(E)),d)+S(\varphi^{-1}(x(E)),d_C^{k,j}).\qedhere$$ 
\end{proof}

\section*{Acknowledgements}
This work began at the summer school ``Lie Algebras, Algebraic Groups and Invariant Theory'' in Samara. We thank the organizers of this summer school
for providing a stimulating environment. 

The research was supported by the grant RSF-DFG 16-41-01013.


\begin{thebibliography}{}
\bibitem[AB]{AB}
V. Alexeev, M. Brion,
{\it Toric degenerations of spherical varieties},
Selecta Math. (N.S.) 10 (2004), no. 4, 453--478.

\bibitem[A]{A}
D.~Anderson, {\it Okounkov bodies and toric degenerations}, 
Math. Ann. 356 (2013), no. 3, 1183--1202.


\bibitem[ABS]{ABS} 
F Ardila, T Bliem, D Salazar, {\it Gelfand-Tsetlin polytopes and Feigin-Fourier-Littelmann-Vinberg 
polytopes as marked poset polytopes},
Journal of Combinatorial Theory, Series A 118 (8), 2454--2462.
	
\bibitem[C]{carter}
R.~Carter, \emph{Lie Algebras of Finite and Affine Type}, Cambridge University Press, New York (2005).

\bibitem[Ca]{caldero}
P.~Caldero, \emph{Toric degenerations of Schubert Varieties}, Transformation Groups, \textbf{7} (2002), no. 1, 51--60.

\bibitem[CF]{CF}
I.Cherednik, E.Feigin,
{\emph  Extremal part of the PBW-filtration and nonsymmetric Macdonald polynomials},
Advances in Mathematics 2015, vol. 282. pp. 220--264.


\bibitem[CFR1]{CFR1} G.~Cerulli Irelli, E.~ Feigin, M.~Reineke, \emph{Quiver Grassmannians and degenerate flag varieties}, 
Algebra \& Number Theory \textbf{6} (2012), no. 1, 165--194. arXiv: 1106.2399.

\bibitem[CFR2]{CFR2} G.~Cerulli Irelli, E.~Feigin, M.~Reineke, \emph{Degenerate flag varieties: moment graphs 
and Schr\"oder numbers}, J. Algebraic Combin. \textbf{38} (2013), no. 1. arXiv:1206.4178.

\bibitem[FFL1]{FFL1} 
Evgeny Feigin, Ghislain Fourier, Peter Littelmann, P
BW filtration and bases for irreducible modules in type $A_n$, Transformation Groups, March 2011, Volume 16, Issue 1, pp 71-89.

\bibitem[FFL2]{FFL2}
E.~Feigin, G.~Fourier, P.~Littelmann,
{\it PBW-filtration and bases for symplectic Lie algebras},
International Mathematics Research Notices, no. 24 (2011), pp. 5760--5784.

\bibitem[FFL3]{FFL3}
E.~Feigin, G.~Fourier, P.~Littelmann,
{\it Favourable modules: Filtrations, polytopes, Newton-Okounkov bodies and flat degenerations},
arXiv:1306.1292, to appear in Transformation Groups.

\bibitem[Fe1]{Fe1}
E.~Feigin, \emph{${\mathbb G}_a^M$ degeneration of flag varieties}, 
 Selecta Mathematica, New Series, vol. 18 (2012), no. 3, pp. 513--537.

\bibitem[Fe2]{Fe2}
E.~Feigin, \emph{Degenerate flag varieties and the median Genocchi numbers}, 
Mathematical Research Letters, no. 18 (6) (2011), pp. 1--16.

\bibitem[Fu]{fult}
W.~Fulton, \emph{Introduction to toric varieties}, Princeton University Press, Princeton (1993).

\bibitem[GZ]{GZ}
I.M.Gelfand, M.L.Cetlin, {\it Finite dimensional representations of the group of 
unimodular matrices}, Doklady Akad. Nauk USSR (N.S.), 71 (1950), 825--828.

\bibitem[G1]{G1}
A.Gornitsky, {\it Essential signatures and canonical bases in irreducible representations of the group $G_2$},
Diploma thesis, 2011 (in Russian).

\bibitem[G2]{G2} A.~Gornitskii,
{\it Essential signatures and canonical bases for irreducible representations of $D_4$},
arXiv:1507.07498.

\bibitem[Ka]{Ka} K.~Kaveh, {\it Crystal  bases  and  Newton-Okounkov  bodies}, 
Duke Math. J. Volume 164, Number 13 (2015), 2461--2506.
   
\bibitem[Ki]{Ki} 
V.~Kiritchenko, {\it Newton-Okounkov polytopes of flag varieties}, Transformation Groups (2016), DOI:10.1007/s00031-016-9372-y

\bibitem[KM]{KM} M.~Kogan, E.~Miller, {\it Toric  degeneration  of  Schubert  varieties and  
Gel’fand-Cetlin polytopes}, Advances in Mathematics, Volume 193, Issue 1, 2005, pp. 1--17. 
		
\bibitem[M]{M}
I.~Makhlin, {\it Brion's Theorem for Gelfand--Tsetlin Polytopes}, Functional Analysis and Its Applications, 50:2 (2016), 98--106

\bibitem[S]{stan}
R.~P.~Stanley, \emph{Two poset polytopes}, Discrete \& Computational Geometry, vol. 1 (1986), no. 1, pp. 9--23.

\bibitem[V]{V}
E.~Vinberg, {\it On some canonical bases of representation spaces of simple Lie algebras},
conference talk, Bielefeld, 2005.

\end{thebibliography}
\end{document}